\documentclass[a4paper,12pt]{article}
\usepackage{a4wide}
\usepackage[utf8]{inputenc}
\usepackage{color}
\usepackage[english]{babel}
\usepackage{amsmath,amsthm,amssymb}
\usepackage{amsfonts}
\usepackage{graphicx}
\newtheorem{thm}{Theorem}
\newtheorem*{mainthm}{Main Theorem}
\newtheorem{cor}{Corollary}
\newtheorem{lem}{Lemma}[section]
\newtheorem{que}[lem]{Question}
\newtheorem{prop}[lem]{Proposition}
\usepackage{varioref}

\theoremstyle{definition}
\newtheorem{ex}[lem]{Example}
\newtheorem{defn}[lem]{Definition}
\newtheorem{rem}[lem]{Remark}
\numberwithin{equation}{section}

\title{Integrability of Continuous Bundles}

\author{Stefano Luzzatto\\ Abdus Salam International Centre for Theoretical Physics (ICTP), Trieste, Italy \\
\texttt{luzzatto@ictp.it}\and Sina Tureli \\ Imperial College London \\ \texttt{sinatureli@gmail.com}\and Khadim War \\ International School for Advanced Studies (SISSA) and \\ Abdus Salam International Centre for Theoretical Physics (ICTP), Trieste, Italy \\ \texttt{kwar@ictp.it}}

\date{August 10, 2016}

\begin{document}

\maketitle

\abstract{We give new sufficient conditions for the integrability and unique integrability of \emph{continuous}
tangent sub-bundles on manifolds of arbitrary dimension, generalizing Frobenius' classical Theorem for \(  C^{1}  \) sub-bundles. Using these conditions we derive new criteria
for uniqueness of solutions to ODE's and PDE's   and for the integrability of invariant bundles
in dynamical systems. In particular we give a novel proof of the Stable Manifold Theorem and  prove some
integrability results for dynamically defined  dominated splittings.
}

\section{Introduction and statement of results}

In this paper we address the question of integrability and unique integrability of \emph{continuous} tangent sub-bundles on $C^r$ manifolds with $r \geq 1$. A continuous $m$-dimensional tangent sub-bundle (or a distribution) $E$ on a $m+n$ dimensional $C^r$ manifold \(  M  \) is a continuous
choice of $m$-dimensional linear subspaces $E_p\subset T_pM$ at each point $p\in M$.
A $C^1$ sub-manifold $N\subset M$ is a local integral manifold of $E$ if $T_pN=E_p$ at each point $p\in N$. The
distribution $E$ is
\emph{integrable}  if there exists local integral manifolds through every point, and \emph{uniquely
integrable} if these
integral manifolds are unique in the sense that whenever two integral manifolds intersect, their
intersection is relatively open in both
integral manifolds.

The question of integrability and unique integrability is a classical problem that goes back to work of Clebsch, Deahna and Frobenius \cite{Cle66, Dea40, Fro77} in the mid 1800's. Besides their intrinsic geometric interest, integrability results have many applications to various areas of mathematics including the existence and uniqueness of solutions for systems of ordinary and partial differential equations and to dynamical systems. The early results develop conditions and techniques to treat cases where  the sub-bundles, or the corresponding equations, are at least \(  C^{1}  \) and, notwithstanding the importance and scope of these results, it has proved extremely difficult to relax the differentiability assumptions completely. Some partial generalizations have been obtained by Hartman \cite{Har02} and other authors \cite{Sim96, MonMor13, Ram07} but all still require some form of weak differentiability, e.g. a Lipschitz condition.

The main point of our results is to formulate new integrability conditions for purely \emph{continuous} equations and sub-bundles. We apply these conditions to obtain results for examples which are not more than H\"older continuous and for which the same statements cannot be obtained by any other existing methods.   Our Main Theorem which contains the most general version of our results is stated in Section \ref{sec-mainthm}. All other results are essentially corollaries and applications of the Main Theorem and are stated in separate subsections of the introduction. We begin with the results for ODE's and PDE's which are of independent interest and easy to formulate.

\medskip\noindent
\emph{Acknowledgments:} Most of the work for this paper was carried out at the Abdus Salam International Center for Theoretical Physics. ST was partially supported by ERC AdG grant no: 339523 RGDD and a major part of the work was completed when he was a Ph.D student in SISSA and ICTP.


\subsection{Uniqueness  of solutions for ODE's}\label{sec-introode}

In this section we consider  continuous ordinary differential equations of the form
\begin{equation}\label{eq-ode1}
\qquad\qquad  \frac{dy^i}{dt} = F^i(t,y(t)), \qquad  (t_0,y(t_0)) \in V
\end{equation}
where $i=1,...,n$ and   $F=(F^1(t,y),...,F^n(t,y)): V\subset \mathbb{R}^{n+1}\rightarrow \mathbb{R}^n$ is a
continuous vector field.
By Peano's theorem, an ODE with a continuous vector field
always admits solutions and by Picard-Lindel\"of-Cauchy-Lipschitz theorem, an ODE with Lipschitz
vector fields admits unique solutions through every point. The Lipschitz condition is however not necessary and a lot of work exists establishing weaker regularity conditions which imply uniqueness, see \cite{AgaLak93} for a comprehensive survey. One such condition is Osgood's criterion (see Theorem 1.4.2 in \cite{AgaLak93})
where the  modulus of continuity $w$ (we give the precise definition below) with respect to the space
variables satisfies
\begin{equation}\label{eq-osgood}
\lim_{\epsilon\rightarrow 0}\int_{0}^{\epsilon}\frac{1}{w(s)}ds < \infty.
\end{equation}
This condition, albeit much weaker than Lipschitz, is also not necessary as there are examples of uniquely integrable ODE's that do not satisfy Osgood's criterion, such as the case $F(t,x) = e^t + x^{\alpha}$ for $\alpha <1$ which was studied in \cite{Ara03}.
 We give here a new condition for uniqueness of solutions for continuous ODE's and a simple  example  of an ODE which satisfies our conditions,  and thus is
 uniquely integrable, but does not satisfy any previously known condition for
 uniqueness.

\begin{defn} Let  $w: \mathbb{R}^{+} \rightarrow \mathbb{R}^{+}$ be an increasing, continuous function such that
$\lim_{s\rightarrow 0}w(s)=0$.
A function $F: U\subset \mathbb{R}^n=(\xi^1,...,\xi^n) \rightarrow \mathbb{R}^m$ is said to
have \emph{modulus of continuity
$w$ with respect to variable} $\xi^i$ if  there exists a constant \( K>0 \) such that for all
$(\xi^1,...,\xi^n)\in U$ and for all $s$ small enough
so that $(\xi^1,...,\xi^i+s,...,\xi^n)\in U$,
$$
|F(\xi^1,...,\xi^i+s,...,\xi^n)-F(\xi^1,...,\xi^n)|\leq K w(|s|).
$$
We say that \(F\) has modulus of continuity \(w\) if it has modulus of continuity \(w\) with respect to all variables.
We denote by $C^{r+ w}$ functions whose $r^{th}$ order partial derivatives have modulus of continuity
$w$ (in the case $r=0$ they will simply be functions with modulus of continuity $w$).
\end{defn}

We denote the extended version of the  vector-field $F(t,x)$  by
$$
 \tilde{F} = \frac{\partial}{\partial t} + \sum_{i=1}^nF^i(t,y)\frac{\partial}{\partial x^i} =
 \sum_{i=1}^{n+1} \tilde{F}^i(\xi) \frac{\partial}{\partial \xi^i}
 $$
where $(\xi^1,...,\xi^{n+1})$ is a collective tag for the coordinates \(  t  \) and $(y^1,...,y^n)$.
Note that
$\tilde{F}(\xi)\neq 0$ for all $\xi$ since $\tilde{F}^1 =1$.

\begin{thm}\label{thm-ode}
Consider the ODE in \eqref{eq-ode1} with $F: V\subset \mathbb{R}^{n+1}$ a continuous  function with modulus of continuity \( w_{1} \). Let
${\xi}\in V$ and $i \in \{1,...,n+1\}$ be such that $\tilde{F}^{i}({\xi}) \neq
0$ and suppose that ${\tilde{F}}$ has modulus
of continuity $w_2$  with respect to variables $(\xi^1, ... , \xi^{i-1},\xi^{i+1},...\xi^{n+1})$ and \begin{equation}\label{eq-mcont2}\lim_{s\rightarrow 0} w_1(s) e^{w_2(s)/s}=0.\end{equation}
Then the ODE \eqref{eq-ode1} with initial condition $(t_0,\xi)$ has a unique solution.
\end{thm}

\begin{ex}\label{ex-1}
Consider the ODE
\begin{equation}\label{eq-ode3}
F(t,x,y)=\left(\frac{dx}{dt},\frac{dy}{dt}\right) = (-tlog(t^{\beta}) - x log(x^{\gamma}),
1+y^{\alpha}-xlog(x^{\delta}))
\end{equation}
for $0<\beta,\gamma,\delta < \alpha <1$ (the requirement that they are less than $1$ is not necessary but we
only impose it to make the examples non-Lipschitz and therefore more interesting). As far as we know, the
uniqueness of solutions for  \eqref{eq-ode3} cannot be verified by any existing method except ours. In
particular, Osgood criterion does not hold due to the term $y^{\alpha}$ and the main result of \cite{Ara03}
would require the right hand side to be Lipschitz with respect to $x$ and $t$ which is not the case.

To see that \eqref{eq-ode3} satisfies the assumptions of Theorem \ref{thm-ode} we will use certain elementary properties of the modulus of continuity, which for convenience we collect in \ref{prop-moc} in the Appendix.  Notice first of all that
 $F(0,0,0) = (0,1)$ so that its $y$ component is non-zero. Moreover by items $4$ and $5$ of Proposition
 \ref{prop-moc}, $F^1$ has modulus of continuity $-slog(s^{\beta})$ with respect to variable $t$ for $s\leq
 \frac{1}{e}$ and has modulus of continuity $-slog(s^{\gamma})$ with respect to variable $x$ for $s\leq
 \frac{1}{e}$ and is constant with respect to $y$; $F^2$ has modulus of continuity $-slog(s^{\delta})$ for
 $s\leq \frac{1}{e}$ with respect to variable $x$ and has modulus of continuity $s^{\alpha}$ with respect to
 variable $y$ and is constant with respect to $t$.
Therefore by Item $7$ of proposition \ref{prop-moc}, letting $\sigma = \max\{\beta,\gamma,\delta\}<\alpha$,
$F$ has modulus of continuity $s^{\alpha}$ with respect to $y$,  modulus of continuity
$w_2(s)=-slog(s^{\sigma})$ with respect to $x$ and $t$,
and finally by Item $6$ of proposition \ref{prop-moc} and the fact that $s^{\alpha} \geq -slog(s^{\sigma})$
for $|s|\leq 1$, and $s^{\alpha}$ is convex, $F$ has modulus of continuity $w_1(s)=s^{\alpha}$.
Since
$
\lim_{s\rightarrow 0}w_1(s) e^{w_2(s)/s} = \lim_{s\rightarrow 0}K x^{\alpha-\sigma}=0
$,  it has unique solutions by Theorem~\ref{thm-ode}.

\end{ex}

\begin{rem}
Theorem \ref{thm-ode} and Example \ref{ex-1} also describe a qualitative way in which regularities between variables can be traded. If a vector field $F$ has its $i^{th}$ component non-zero then you can decrease the vector field's
regularity with respect to the $i^{th}$ variable as long as you increase the others in a way described by
equation \eqref{eq-mcont2}. It is a more flexible criterion than both Osgood and that of \cite{Ara03}.
Condition \eqref{eq-mcont2} is satisfied for example if the vector field is only H\"older continuous
overall, i.e. $w_1(s) = s^{\alpha}$, but, restricting to all but one variables is just a little bit better
than H\"older continuous, such as for example  $w_2(s) =
-slog(s^{\beta})$  for some \( \beta \in (0, \alpha) \).
\end{rem}

\begin{rem}
Part of the interest in condition \eqref{eq-mcont2} lies in the fact that two different regularities \(
w_{1}, w_{2} \) to come into play. Clearly we  always have \( w_{2}\leq w_{1} \) and therefore Theorem
\ref{thm-ode} holds under the stronger condition obtained by using the overall regularity, i.e. letting \( w=w_{1}=w_{2} \) to get
\begin{equation}\label{eq-mcont3}
\lim_{s\rightarrow 0} w(s) e^{w(s)/s}=0.
\end{equation}
Thus, as an immediate corollary of Theorem \ref{thm-ode}, \emph{a  vector field \( F \) with modulus of continuity \( w \) such that
$w$ satisfies equation \eqref{eq-mcont3} is uniquely integrable}.
In view of this, it is natural to search for, and try to describe and characterize,
functions \( w \) satisfying condition \eqref{eq-mcont3} and in particular to compare this condition with Osgood's. One can check that many functions
$w$ verify both \eqref{eq-osgood} and  \eqref{eq-mcont3}, such as $w_1(s) = slog^{1+s}(s)$ or $w_1(s) = sloglog....(s)$, and many
others satisfy neither condition, such as $w_1(s)=s^{\alpha}log(s)$ for $\alpha<1$. So far we have not
however been able to show that the two conditions are equivalent nor to find any examples of functions which
satisfy one and not the other.  In any case, this simplified version also gives an interesting way to
replace Osgood's criterion with a relatively easy limit condition, at least for the most relevant examples
that we know.
\end{rem}

\begin{que}
Are conditions \eqref{eq-osgood} and \eqref{eq-mcont3} equivalent?
\end{que}

\begin{rem} Added in proof. We are grateful to Graziano Crasta for pointing out to us that  \eqref{eq-mcont3} implies 
\eqref{eq-osgood}. This implication can be obtained by studying the behavior of   $\frac{w(s)}{s} e^{w(s)/s}$ using an asymptotic expansion of the Lambert function $W(x)$, which is the function that is the solution to $W(x)e^{W(x)}=x$, as $x$ goes to infinity. One then gets that \eqref{eq-osgood} is satisfied if and only if the function 
$[s log (\frac{\rho(s)}{s})]^{-1}$ is not integrable on some right neighborhood of $0$. One can then show that \eqref{eq-mcont3} implies indeed that it is not. We are still not aware if the reverse implication is true. 
\end{rem}

\subsection{Uniqueness of solutions for PDE's}\label{sec-intropde}

In this section we consider  linear partial differential equations of the form
\begin{equation}\label{eq-pde1}
\qquad\qquad \frac{\partial y^i}{\partial x^j} = F^{ij}(x,y(x)), \qquad (x,y(x)) \in V
\end{equation}
where $i=1,...,n$, $j=1,...,m$ and $F^{ij}: V\subset \mathbb{R}^{n+m} \rightarrow \mathbb{R}$ are continuous
functions.
Note that the ODE \eqref{eq-ode1} is a special case of \eqref{eq-pde1}, the case where $m=1$. We again
denote the collective coordinates
$(\xi^1,...,\xi^{n+m})=(x^1,...,x^m,y^1,...,y^n)$.
In this case we define the \(  n\times(m+n)  \) matrix extension \( \hat F \) of the $n\times m$ matrix  $F^{ij}$ by
$$
\hat{F}^{ij}(\xi) = \delta_{ij} \ \text{for} \ 1\leq i \leq n, \ 1\leq j \leq m
$$
and
$$
\hat{F}^{ij}(\xi) = F^{ij}(\xi)  \ \text{for} \ 1\leq i \leq n, m+1\leq j\leq m+n.
$$
Given any set of indices $I=(i_1,...,i_n)$,  we
denote the submatrix
of $\hat{F}(\xi)$ which corresponds to the $i_1,...,i_n$'th columns by $\hat{F}^{I}(\xi)$.

The existence of solutions for PDE's is not automatic as it is in the ODE setting, and in particular is not a direct consequence of their regularity, and so the following result concerns the uniqueness of solutions while  assuming  their existence. Our most general results below also give conditions for existence of solutions, but they require a more geometric ``involutivity'' condition which is independent of the regularity and thus not so easy to state in this setting.

\begin{thm}\label{thm-pde}Consider the PDE in  \eqref{eq-pde1} with $F^{ij}$ continuous with modulus of continuity \( w_{1} \).
Let ${\xi} \in V$ be a point such that
for some $I=(i_1,...,i_n)$, $det(\hat{F}^{I}({\xi})) \neq 0$ and suppose that
$F^{ij}$ has modulus of continuity $w_2$ with respect to the variables $\{\xi^{i_1},...,\xi^{i_n}\}$
and that
\begin{equation}
\label{eq-mcont4}\lim_{s\rightarrow 0} w_1(s) e^{w_2(s)/s}=0.
\end{equation}
Then if the PDE
\eqref{eq-pde1} admits a solution
at $\xi$, that solution is unique.
\end{thm}

\begin{rem}
Notice that if  \(  I=(m+1,...,m+n)  \) then \(  det(\hat F^{I}(\xi))=1\neq0  \) and therefore if $F^{ij}(x,y)$ has modulus of continuity $w_2$ with respect to variables $(y^1,...,y^n)$
so that \eqref{eq-mcont4} is satisfied, then  the solutions of \eqref{eq-pde1}, whenever they exists, are
unique.
\end{rem}

Constructing examples of PDE's satisfying our conditions is more complicated than constructing examples of
ODE's because of the problem of existence of solutions mentioned above. We therefore formulate our example within a special class of equations for which existence can be verified directly.
 Suppose the
functions \( F^{ij} \)  are of the form
\begin{equation}\label{eq-special}
F^{ij}(x,y) = G_i(y^i)\frac{\partial H_{i}}{\partial x^j}(x)
\end{equation}
with  $(x,y)\in V=V_1 \times V_2$, for continuous functions $G_{i}: U_i\subset
\mathbb{R} \rightarrow \mathbb{R}$, $H_{i}:V_2\subset \mathbb{R}^m \rightarrow \mathbb{R}$ for $j=1,...,m$,
$i=1,...,n$, $V_1 = U_1 \times ... \times U_n$. In this case the next proposition (proved in Section \ref{sec-ex}) tells us that we can get existence of solutions  with no additional regularity assumptions.

\begin{prop}\label{pro-pde}
Consider a partial differential equation \eqref{eq-pde1} where the functions \( F^{ij} \) are of  the form
\eqref{eq-special}. Then \eqref{eq-pde1} admits solutions through every point.
\end{prop}

This allows us to give examples of PDE's to which we can apply our uniqueness criterion.
\begin{ex} Consider the PDE
\begin{equation}\label{eq-pde2}\frac{\partial y^i}{\partial x^j}(x,y) =
-(x^j)^{\alpha_{ij}}y^ilog((y^i)^{\beta_i})\end{equation}
for  parameters $0 < \beta_i, \alpha_{ij} <1$ and also $  \beta = \max_{i}\beta_{i} < \alpha =
\min_{i,j}\alpha_{ij}$.
This equation can be written in the form \eqref{eq-special} with
$$
H_{i}(x) = \sum_{j=1}^n\frac{(x^j)^{\alpha_{ij}+1}}{{\alpha_{ij}+1}}
\quad\text{ and } \quad
G_i(y^i) = -y^ilog((y^i)^{\beta_i})
$$
Therefore by Proposition \ref{pro-pde} it admits solutions. For uniqueness, we have that the regularity of the $H_{i}$'s
is $C^{1+w_3}$ with $w_3(s) = s^{\alpha}$
and the regularity of the  $G_i$'s is $C^{w_2}$ with $w_2(s) = -slog(s^{\beta})$. Therefore, letting $w_1(s) = \max \{w_2(s),w_3(s)\}$, we have
$\lim_{s\rightarrow 0}w_1(s) e^{w_2(s)/s} = \lim_{s\rightarrow 0}K x^{\alpha-\beta}=0$ and using Theorem \ref{thm-pde} we have that
this system has unique solutions.
\end{ex}

\begin{ex}
Consider the PDE
\[
\begin{aligned}\label{eq-pde3}
\frac{\partial y^1}{\partial x^1}(x,y) &= (1-x^1ln((x^1)^{\alpha_{11}}))((y^1)^{\beta_1}+1)\\
\frac{\partial y^1}{\partial x^2}(x,y) &= (x^2)^{\alpha_{12}}(1+(y^1)^{\beta_1})\\
\frac{\partial y^2}{\partial x^1}(x,y) &= y^2ln((y^2)^{\beta_{2}})x^1ln((x^1)^{\alpha_{21}})\\
\frac{\partial y^2}{\partial x^2}(x,y) &= -y^2ln((y^2)^{\beta_{2}})(x^2)^{\alpha_{22}}
\end{aligned}
\]
with $0<\alpha_{12},\alpha_{22},\beta_{1} < \alpha_{11}, \alpha_{21},\beta_2<1 $.
Set $\alpha = \max\{\alpha_{11},\alpha_{21},\beta_2\}$ and $\beta =
\min\{\alpha_{12},\alpha_{22},\beta_1\}$. The right-hand side again has the form \eqref{eq-special} and so the PDE has solutions by Proposition \ref{pro-pde}.
For the uniqueness, one considers the matrix
$F_{ij}$ at $(x,y)=(0,0)$, which is
\[
\left(
  \begin{array}{cccc}
    1 & 0 & 1 & 0 \\
    0 & 1 & 0 & 0 \\
  \end{array}
\right)
\]
and so the sub-matrix corresponding to columns $i_1=2$ and $i_2=3$ is invertible. But then with respect to
variables $\xi^{2}=x^2$ and $\xi^3=y^1$, $F^{ij}$ has modulus of continuity $w_2(s)=-sln(s^{\alpha})$ and in
general has modulus of continuity $w_1(s) = s^{\beta}$. But $w_1(s)$ and $w_2(s)$ satisfy condition
\eqref{eq-mcont4} and so, by Theorem \ref{thm-pde}, the PDE has unique solutions in a neighborhood of the origine \((0,0)\).
\end{ex}

These are only two particularly simple examples one can construct using Proposition \ref{pro-pde}. Here the forms of $H_i(y)$ are quite simple
 in the sense that $F_i(y) = \sum_{j} G_{ij}(y^j)$ and more complicated examples can
be achieved with more work.

\subsection{Unique Integrability of Continuous Bundles}\label{sec-introbundle}

We will derive the results above on existence and uniqueness of solutions for ODE's and PDE's from more
general and more geometric results about the integrability and unique integrability of tangent bundles on
manifolds. In this section we state Theorem \ref{thm-alt2} which can be seen as a mid step between the ODE and PDE theorems stated in the previous sections and the more general  results in  Theorem
\ref{thm-main} and the Main Theorem in the following sections.
Throughout this section we assume that $E$ is a continuous tangent sub-bundle on a manifold $M$.
First we need to generalize certain classical definitions of modulus of continuity to bundles.

\begin{defn}\label{def-transmod}
A bundle $ E$ of rank \(  m  \) is said to have modulus of continuity $w$, where $w$ is a continuous, increasing function $w: I \subset \mathbb{R}^+ \rightarrow \mathbb{R}^+$ such that $\lim_{s\rightarrow 0} w(s) =0$, if in every sufficiently small neighbourhood, it can be spanned by linearly independent vector fields
$X_1,...,X_m$ such that in local coordinates $|X_i(p) - X_i(q)| \leq w(|p-q|)$.
 $E$ is said to have \emph{transversal modulus of continuity} $w$ if for every $x_0 \in M$, there exists a coordinate neighbourhood $(V,(x^1,...,x^m,y^1,...,y^n))$ around $x_0$ so that $E$ is transverse to $\text{span}\{\frac{\partial}{\partial y^i}\}_{i=1}^n$ and with respect to coordinates $\{y^1,...,y^n\}$ $E$ has modulus of continuity $w(s)$.
\end{defn}

\begin{thm}\label{thm-alt2}
Let $E$ be a rank \(  n  \) bundle with modulus of continuity \(  w_{1}  \) and transversal modulus of continuity $w_2$. Assume $E$ is integrable and
\begin{equation}\label{eq-mcont5}
\lim_{s\rightarrow 0} w_1(s) e^{w_2(s)/s}=0
\end{equation}
Then $E$ is uniquely integrable.
\end{thm}

The scope of Theorem \ref{thm-alt2} is possibly limited because integrability is assumed. However we have stated it here because it gives uniqueness as a result of a natural regularity condition and in general  existence is a highly non trivial property which cannot be reduced to any regularity condition. We will show that Theorem \ref{thm-alt2} is a corollary of the more general results below which address the problem of existence as well as uniqueness, and that it easily implies Theorems \ref{thm-ode} and \ref{thm-pde}.

\subsection{Asymptotic involutivity and exterior regularity}
We now formulate the special case of our most general result, addressing the problem of existence and uniqueness of integral manifolds for continuous tangent bundle distributions. Since integrability of tangent bundles is a local property, we assume from now on that $U$ is a  Euclidean ball in $\mathbb{R}^{n+m}$ and \( E \) is a continuous tangent bundle distribution  of rank \( n \) on \( U \). We let $|\cdot|$ denote the (induced) Euclidean norm on sections of the tangent bundle and of $k$-differential forms for all $0 \leq k \leq n+m$. $|\cdot|_{\infty}$ denotes the sup-norm over $U$, which gives the aforementioned sections a Banach space structure. We also employ , point-wise, tangent vectors with the induced Euclidean metric. Letting \( \mathcal A^{1}(E) \) denote the space of all 1-forms \( \eta \) defined on \( U \) with \( E\subset \ker (\eta) \),  the distribution \( E \) is completely described by any set $\{\eta_{i}\}_{i=1}^n$  of \( n \) linearly independent 1-forms in, i.e. any basis of, \( \mathcal A^{1}(E) \). If the distribution \( E \) is differentiable, the forms $\{\eta_{i}\}_{i=1}^n$ can also be chosen differentiable and the classical Frobenius theorem \cite{Cle66, Dea40, Fro77} states that $E$
is \emph{uniquely integrable} if, for any basis of differentiable \(1\)-forms $\{\eta_{i}\}_{i=1}^n$ of $\mathcal{A}^1(E)$,  the \emph{involutivity condition}
\begin{equation}\label{eq-Fro}
|\eta_{1} \wedge\cdots\wedge \eta_{n}\wedge d\eta_{i}(p)|=0
\end{equation}
holds for all $i=1,...,n$ and $p \in U$. Several generalizations of this Theorem exist in the literature, including results which weaken the differentiability assumption, we mention for example results by Hartman \cite{Har02}, Simic \cite{Sim96} and other authors \cite{MonMor13, Ram07}, but which still essentially use the fact that the exterior derivative \( d\eta_{i} \) exists, for example if \( E \) is Lipschitz then the \( \eta_{i} \) are differentiable almost everyhere and \( d\eta_{i} \) exists almost everywhere, and therefore such results can be formulated in essentially the same way as Frobenius, using condition \eqref{eq-Fro}.

One of the first stumbling blocks in obtaining some integrability criteria for general \emph{continuous} distributions is that the exterior derivatives of the forms  $\{\eta_{i}\}_{i=1}^n$ which define \( E \) do not in general exist and it is thus not even possible to state condition~\eqref{eq-Fro}. Our strategy for resolving this issue is to consider a sequence $\{\eta_{i}^{k}\}_{i=1}^n$ of \( C^{1} \) differential forms, for which therefore the exterior derivatives \( d\eta_{i}^{k} \)'s do exist,  which converge to $\{\eta_{i}\}_{i=1}^n$ and satisfy certain conditions which we define precisely below and which imply that the sequence is in some sense ``asymptotically involutive'' and which will allow us to deduce that \( E \) is integrable without having to define an involutivity condition directly on \( E \). A quite interesting by-product of this approach is a clear distinction between the involutivity conditions required for integrability and the regularity conditions required for unique integrability.
In the \( C^{1} \) case these regularity conditions are automatically satisfied and thus the involutivity condition~\eqref{eq-Fro} is sufficient to guarantee both integrability and unique integrability.

\begin{defn}\label{defn-asyminv}A continuous tangent sub-bundle \(E\) of rank \(n\)  is \emph{strongly asymptotically involutive}
if there exist a basis $\{{\eta}_i\}_{i=1}^n$ of $\mathcal{A}^1(E)$, a constant  $\epsilon_0>0$, and a sequence of $C^1$ differential
1-forms $\{\eta^k_i\}_{i=1}^n$ such that $\max_i|\eta^k_i - \eta_i|_{\infty}\rightarrow 0$ as $k \rightarrow \infty$ and
\begin{equation}\label{eq-asyinv}
\max_j|\eta^k_1 \wedge ... \eta^k_n \wedge d\eta^k_j|_{\infty} e^{\epsilon_0\max_i|d{\eta}^k_i|_{\infty}} \rightarrow 0 \quad
\text{as $k \rightarrow \infty$}.
\end{equation}
\end{defn}

\begin{defn}
A continuous tangent sub-bundle \(E\) of rank \(n\)  is  \emph{strongly exterior regular} if there exist a basis $\{{\beta}_i\}_{i=1}^n$ of $\mathcal{A}^1(E)$, a constant $\epsilon_1>0$, and a sequence of $C^1$ differential
1-forms $\{\beta^k_i\}_{i=1}^n$ such that
\begin{equation}\label{eq-extreg}
\max_j|\beta^k_j - \beta_j|_{\infty} e^{\epsilon_1\max_{i}|{d\beta}^k_i|_{\infty}}
\rightarrow 0 \quad \text{as $k\rightarrow\infty$}.
\end{equation}
\end{defn}

We note that we refer to these conditions as ``strong'' since we will define some more general versions below.

\begin{thm}\label{thm-main} Let $E$ be a continuous tangent subbundle.
If \(E\) is strongly asymptotically involutive then it is integrable.
If \(E\) is integrable and strongly exterior regular then it is uniquely integrable.
\end{thm}
Notice that if \(E\) is \( C^{1} \), the strong exterior regularity is trivially satisfied by choosing \(\beta^{k}_{i}=\eta_{i}\) and  the Frobenius involutivity condition \eqref{eq-Fro} is equivalent to the strong asymptotic involutivity condition \eqref{eq-asyinv} by choosing \(\eta^{k}_{i}=\eta_{i}\).  We remark that a version of Theorem \ref{thm-main} in dimension \(  \leq 3  \) was obtained in \cite{LuzTurWar16} by using different arguments.

\begin{rem}
One can also combine the asymptotic involutivity condition of Theorem \ref{thm-main}, which gives integrability, with the condition on the modulus of continuity of \( E \) in Theorem \ref{thm-alt2} which then gives uniqueness (indeed, we will show below that the condition of Theorem \ref{thm-alt2} implies exterior regularity) as this last condition may be easier to check in some situations.  As we discuss in more details below, conditions such as those of asymptotic involutivity and exterior regularity, which are based on a sequence of approximations, are actually quite natural. It would also be interesting however to know whether there is any way to formulate the existence conditions without recourse to approximations, directly in terms of properties of the bundle E (or $\mathcal{A}^1(E)$ to be more precise).  \end{rem}

\begin{que}Can the strong asymptotic involutivity condition in Theorem \ref{thm-main} be replaced by a condition that can be stated only in terms of geometric and analytic properties of the bundle $E$ rather than a sequence of approximations? \end{que}

An answer to this question would be a natural form of Peano's Theorem in higher dimensions.

\subsection{The Main Theorem}\label{sec-mainthm}
 In this section we state our most general theorem, which contains Theorem \ref{thm-main} as a special case and also implies all the other results stated above. This more general result will be important for applications to the tangent bundles which arise in the context of Dynamical Systems.

Let \(U\subset\mathbb{R}^{m+n}\). Given two tangent bundles $E^1$ and $E^2$ on $U$, for $x\in U$, we denote by $\angle(E^1_x,E^2_x)$  the maximum angle between all possible  rays  \(  R_{1}\subset E^{1}, R_{2}\subset E^{2}  \)  orthogonal to \(  E^{1}_{x}\cap E^{2}_{x}  \)  (with respect to the induced point-wise metric at $x$). A sequence of bundles $\{E^k\}$  is said to converge to $E^1$ in angle if $\sup_{x\in U} \angle(E^1_x,E^k_x)\to 0$ as $k\to\infty$.
This also means that the Haussdorff distance between the unit spheres inside $E^k_x$ and $E^1_x$ goes to zero for all $x$.

Now assume we are given \(E\) a continuous tangent bundle of rank \(n\) on \(U\). 
We choose a coordinate system $(x^1,...,x^m,y^1,...,y^n)$ in \(U\) so that the $y^i$ coordinates are transverse to $E$ and  if $E^k$  is any sequence of bundles of rank \(n\) which converge in angle to $E$, then we can assume without loss of generality that they are also transverse to the \(y^{i}\) coordinates. We denote the subspace
$$
\mathcal{Y}_p:=\text{span}\left\{\frac{\partial}{\partial y^1},...,\frac{\partial}{\partial y^n}\right\}|_p.
$$
By the transversality condition, we can span $E$ by vectors of the form
$$
X_i = \frac{\partial}{\partial x^i} + \sum_{j=1}^n a_{ij}(x,y)\frac{\partial}{\partial y^j}
$$
for some $C^0$ functions $a_{ij}(x,y)$,
and if  $E^k$ is a sequence of \(C^{1}\) bundles converging to \(E\) then each \(E^{k}\) can be spanned  by vectors of the form
$$
X^k_i = \frac{\partial}{\partial x^i} + \sum_{j=1}^n a^k_{ij}(x,y)\frac{\partial}{\partial y^j}
$$
 for some sequence of $C^1$  functions $a^k_{ij}(x,y)$ so that $X^k_i$ converges to $X_i$ as $k\rightarrow \infty$.

Note that a basis of sections $\{\alpha^{k}_{1},...,\alpha^{k}_{n}\}$ of $\mathcal{A}^1(E^k)$ defined
on $U$, gives a non-vanishing section of the frame bundle $F(\mathcal{A}^1(E^k))$ of $\mathcal{A}^1(E^k)$.  We denote this section by $A^{k}$, which in local coordinates is the matrix of 1-forms whose $j^{th}$ row is $\alpha^{k}_j$. More explicitly if evaluated at a point $p$ it is the map
$A^{k}_p : \mathbb{R}^{m+n} \rightarrow \mathbb{R}^n$ defined by
$$
A^{k}_p(v):= (\alpha^{k}_{1}(v),...,\alpha^{k}_n(v))|_p.
$$
Sometimes if we evaluate it along a curve $\gamma$ then we denote $A^{k}_s = A^{k}_{\gamma(s)}$. By our assumptions $A^{k}_p$ has rank $n$ (since it has $n$ rows made from a linearly independent set of 1-forms) and $ker(A^{k}_p) =E^k_p$. Therefore restricted to $\mathcal{Y}_p$ which is transverse to $E^k_p$
these maps are invertible and we write
\[
A^{-k}_p:=(A^{k}|_{\mathcal Y_{p}})^{-1}.
\]
In the statement and proof of the theorem, we will use a sequence of open covers $\{U^{k,i}\}_{i=1}^{s_{k}}$ of \(U\) associated to sequence \(E^{k}\) of approximating bundles and a corresponding sequence of sections $\{A^{k,i}\}_{i=1}^{s_{k}}$ defined on the elements of these covers. We will use the notation  \(
A^{-k,j}_p:=(A^{k,j}|_{\mathcal Y_{p}})^{-1}. \)
Since the elements of these covers overlap we will need the following compatibility condition.

\begin{defn} A finite open cover $\{U^{k,i}\}_{i=1}^{s_{k}}$ of \(U\) is a \textit{compatible cover} for non-vanishing sections $\{A^{k,i}\}_{i=1}^{s_{k}}$ of the frame bundle $F(\mathcal{A}^1(E^{k}))$ defined on \(U^{k,i}\) if
for all \(i,j=1,...,s_{k}\),  $p \in U^k_i \cap U^k_j$ we have
 $$||A^{k,i}_p \circ A^{-k,j}_p||=1.$$
\end{defn}
Given a compatible cover, we also define the maps $dA^{k,i}:\mathbb{R}^{2(n+m)} \rightarrow \mathbb{R}^n$ by
$$
dA^{k,i}_p(u,v) = (d\alpha^{k,i}_{1}(u,v),...,d\alpha^{k,i}_n(u,v))|_p
$$
for $u,v \in \mathbb{R}^{n+m}$. We denote by $dA^{k,i}|_{E^k_p}$ the restriction of this map to $E^k_p\times E^{k}_p$.
We also define the following constant depending on $k$ and \(i\)
\begin{equation}\label{def-Mk}
M^{k,i}_{A}:= \sup_{\substack{v \in E, w \in \mathbb{R}^{n} \\ |v|=|w|=1 \\ p \in U^{k,i} }} |dA^{k,i}_p(A^{-k,i}_pw,v)|.
\end{equation}

\begin{defn}\label{defn-asyminv2} A continuous tangent subbundle $E$ on $U \subset \mathbb{R}^{n+m}$  is \textit{ asymptotically involutive} if there is a sequence of $C^1$ subbundles $E^k$ that converge to $E$, \( \epsilon>0 \) and, for all~$k$, a compatible open cover $\{U^{k,i}\}_{i=1}^{s_k}$ of $U$
with non-vanishing sections $\{A^{k,i}\}_{i=1}^{s_{k}}$ of $F(\mathcal{A}^1(E^k))$ defined on $U^{k,i}$ such that
$$\max_{\substack{i,j,\ell\in \{1,...,s_k\}}}||dA^{k,i}|_{E^k}||_{\infty} \ ||A^{-k,j}||_{\infty} \ e^{\epsilon  M^{k,\ell}_{A}}\rightarrow 0 \ \text{ as } k \rightarrow 0.
$$
\end{defn}

\begin{defn}\label{defn-extreg2} A continuous tangent subbundle $E$ on $U \subset \mathbb{R}^{n+m}$  is \textit{exterior regular} if there is a sequence of $C^1$ bundles $E^k$ that converge to $E$,   \( \epsilon>0 \) and, for all $k$ a compatible open cover $\{U^{k,i}\}_{i=1}^{s_k}$ of $U$
and non-vanishing sections $B^{k,i}$ of $F(\mathcal{A}^1(E^k))$ defined on $U^{k,i}$ such that
 $$\max_{\substack{i,j,\ell \in \{1,...,s_k\} }}\|B^{k,i}|_{E}\|_{\infty} \ ||B^{-k,j}||_{\infty} \ e^{\epsilon M^{k,\ell}_{B}}\rightarrow 0 \ \text{as } k \ \rightarrow 0.
 $$
\end{defn}
\begin{mainthm}\label{thm-main2}
Let \( E \) be a continuous tangent subbundle. If $E$ is asymptotically involutive then $E$ is integrable. If $E$ is integrable and exterior regular then it is uniquely integrable
\end{mainthm}

\begin{rem}
The proof of Theorem~\ref{thm-main} consists of verifying that the strong asymptotic involutivity and strong exterior regularity conditions are simply  special cases of their more general versions given here. There are two main differences which make the  general versions more general, and more applicable, than the strong versions. The first is that in the general versions of asymptotic involutivity and exterior regularity the forms defining the sub-bundles are only defined locally.  The second, more important, difference is that in  the the strong versions, the differential forms $\{\eta^k_1,...,\eta^k_n\}$ are assumed to converge to a set of linearly independent
forms, whereas this is not required by the general version.
Indeed, multiplying a form by a constant or even by a function, does not change its kernel and thus does not change the bundle that it defines, and what one really needs is the convergence of a sequence of approximating bundles not necessarily the forms defining these bundles. Thus assuming the convergence of the forms, while allowing for a tidier formulation of the conditions, is an unnecessary restriction.
This more general formulation allows us in particular to obtain an application to dynamical systems, including  the  well known Stable Manifold Theorem, which would not follow from Theorem \ref{thm-main}.

\end{rem}
\subsection{Stable Manifold Theorem}
A rich supply of continuous, integrable and non-integrable distributions come from dynamical systems where some dynamically defined tangent bundles occur naturally. The integrability (or not)  of these subbundles has implications for the study of statistical and topological properties of such systems \cite{BonWil, Ham, HamPot} and there is a rich literature going back to Hadamard and Perron \cite{Had01, Per29, Per30} concerning techniques for studying the problem, see also \cite{BriPes74,  HirPugShu77, Irw70, Pes76} for classical  results going back to the 1970's and \cite{BriBurIva, PotHam, LuzTurWar15b, Pes, HerHerUre2} for an overview of recent approaches. We give here a fairly general class of dynamical systems, which in particular includes classical uniform hyperbolic systems and certain partially hyperbolic systems, for which the assumptions of the Main Theorem can be readily verified. This gives a unification of many results, which have so far been proved by a variety of techniques, as a direct corollary of a single abstract Frobenius type integrability
result.

Throughout this section $M$ denotes an \((n+m)\)-dimensional compact manifold  and $\phi: M\to M$ denotes a $C^2$ diffeomorphism.
The diffeomorphism $\phi$ is said to admit a \textit{dominated splitting} if there exists a \(D\phi\)-invariant continuous decomposition $E\oplus
F$ of $TM$  such that
\begin{equation}\label{eq-dominated}
\sup_{x\in M}\|D\phi_{x}|_{E_{x}}\|<\inf_{x\in M}m(D\phi_{x}|_{F_{x}}).
\end{equation}
Here $m(\cdot)$ denotes the conorm of an operator, that is $m(D\phi|_F(x))=\inf_{v\in F(x)}\frac{|D\phi v|}{|v|}$.

Note that \eqref{eq-dominated} is a purely dynamical condition and there is no a priori reason why such condition, or any other similar dynamical condition, should have a bearing on the question of integrability. However,
 remarkably, stronger domination conditions such as uniform hyperbolicity, where \(\|D\phi|_{E}\|<1<m(D\phi|_{F})\), do imply integrability of both subbundles \cite{HirPugShu77}, though there are counterexamples which show that weaker dominated splittings as in \eqref{eq-dominated} do not \cite{Sma67, Wil98} and also that systems with  dominated splitting may be integrable but not uniquely \cite{HerHerUre}.
We give here a sufficient condition for unique integrability for a class of systems with dominated splitting which contains the uniformly hyperbolic diffeomorphisms but significantly relaxes the contraction of the subbundle \(E\) to allow for neutral  or mildly expanding behavior (including, for example, the time one map of uniformly hyperbolic flows).

\begin{defn}\label{defn-linearly}
$E$ is called at most \textit{linearly growing} for \(\phi\) if there exists constants $C,D$ such that $|D\phi^k|_{E(x)}|\leq kC + D$ for all $x \in M$ and $k\geq0$.
\end{defn}

\begin{thm}\label{thm-dyn}
Let \(\phi: M\to M\) be a \(C^{2}\) diffeomorphism with an invariant dominated splitting \(E\oplus F\).
If $E$ is at most linearly growing  then \(E\) is uniquely integrable.
\end{thm}

A particular case of diffeomorphisms with dominated splitting are  \textit{partially hyperbolic systems}, which have a \(D\phi\)-invariant splitting \(E^{s}\oplus E^{c}\oplus E^{u}\) where
\[
\|D\phi|_{E^{s}}\|<1<m(D\phi|_{E^{u}}) \ \text{ and } \
\|D\phi|_{E^{s}}\|<m(D\phi|_{E^{c}})\leq\|D\phi|_{E^{c}}\|<m(D\phi|_{E^{u}}).
\]
\begin{cor}\label{thm-part}
Let \(\phi: M\to M\) be a \(C^{2}\) partially hyperbolic diffeomorphism then if \(E^{c}\) grows at most linearly for \(\phi\) and \(\phi^{-1}\) then it is uniquely integrable.
\end{cor}
Corollary \ref{thm-part} generalizes a result in \cite{Bri03} that gives unique integrability for \(E^{c}\) under the stronger assumption that \(\phi\) is \textit{center-isometric}, i.e \(\|D\phi v\|=\|v\|\) for every \(v\in E^{c}\).
Note that partially hyperbolic systems are special cases of dominated splitting in \eqref{eq-dominated} where \(E=E^{s}\oplus E^{c}\) and \(F=E^{u}\) or \(E=E^{s}\) and \(F=E^{c}\oplus E^{u}\). Therefore  Corollary \ref{thm-part} is a direct application of Theorem \ref{thm-dyn}, by showing that both \(E^{s}\oplus E^{c}\) and \(E^{c}\oplus E^{u}\) are uniquely integrable.

\subsection{ Philosophy and overview of the paper}\label{sec:phil}
Our main result is the Main Theorem, whose proof will occupy Sections \ref{subsec-existence} and \ref{subsec-uniqueness}, and all other results are, directly or indirectly, corollaries of the Main Theorem and will be proved in Section \ref{sec-ex}.
In the Appendix we prove some basic lemmas required from analysis.
The proof of the Main Theorem can be divided, as usual, into two parts: The existence of integral manifolds, which will be proved in Section \ref{subsec-existence},
and the uniqueness, which will be proved in Section \ref{subsec-uniqueness}.

The key idea in the proof of existence is the following.  Given a set of \(m\) linearly independent differentiable vector fields \(X_{1},...,X_{m}\), there is a canonical way of constructing an \(m\)-dimensional manifold \(W\) by successive integration of the vector fields, see \eqref{W}. In the case where the Frobenius involutivity \eqref{eq-Fro} is satisfied, \(W\) can be shown to define an integral manifold of the span of \(X_{1},...,X_{m}\), and this is indeed one possible strategy to prove Frobenius theorem.
Our main idea is to give a quantitative estimate of how ``non-integrable'' the manifold \(W\) is in the general case in terms of certain quantities which come into our definition of asymptotic involutivity, see Proposition \ref{prop-uniformsurfaces}. We then apply Proposition \ref{prop-uniformsurfaces} to our sequence \(E^{k}\) to get that the corresponding manifolds \(W^{k}\) are getting closer to being integral manifold and we show that the limit defines an integral manifold, see Section \ref{sec:convergence}.

The proof of Proposition \ref{prop-uniformsurfaces} relies on the crucial observation that the involutivity is essentially related to the pushforward of vector fields along flows. Indeed, one way to write the involutivity of a bundle \(E\)  is
that there is a choice of vector fields \(X_{1},...,X_{m}\) that span \(E\) such that \([X_{i},X_{j}]=0\) or, equivalently,  that the pushforward along the flow of \(X_{i}\) leaves \(X_{j}\) invariant, i.e
\[
[X_{i},X_{j}]=0\Longleftrightarrow De^{tX_{i}}X_{j}=X_{j}
\]
where \(e^{tX_{i}}\) denotes the flow of \(X_{i}\). The quantitative measurement of non-integrability of \(E\) mentioned above is thus essentially given by the quantity \(De^{tX_{i}}X_{j}-X_{j}\). This difference can further be expressed by the pushforward of the Lie bracket \([X_{i},X_{j}]\) along the flow of \(X_{i}\), see \eqref{pusg}, which reduces the problem of that of estimating the norm of the pushforward.

The method by which we achieve this is perhaps the main technical innovation in the paper.
Standard techniques give estimates of the form
\begin{equation}\label{eq-pushest}
||De^{tX}_p||_{\infty} \lesssim e^{t|X|_{C^1}}.
\end{equation}
However this is not useful when $X$ approximates a continuous vector field, as in our case, since the $C^1$ norm  of $X$ might blow-up. In certain settings, using the notation of differential forms, there is a better estimate by  Hartman (see section 9 of chapter 5 in \cite{Har02}), who gives
\begin{equation}\label{eq-pushest2}
||De^{tX}_p|| \lesssim e^{t|d\eta|_{\infty}}
\end{equation}
where \(X\in\ker{\eta}\). It is easy to see that \eqref{eq-pushest} is much weaker than \eqref{eq-pushest2}. For example we consider the  simple case where \(X=\partial_{x}+b\partial_{y}\) and \(\eta=dy-bdx\).
In this case, \(|X|_{C^{1}}\) involves both $|\frac{\partial b}{\partial x}|_{\infty}$ and $|\frac{\partial b}{\partial y}|_{\infty}$ whereas \(|d\eta|_{\infty}\) involves only  $|\frac{\partial b}{\partial y}|_{\infty}$ since $d\eta = \frac{\partial b}{\partial y} dx\wedge dy$.
Another example is where \(\eta=df\) for some \(C^{1}\) function \(f\) and \(X\) is any vector field in the kernel of \(\eta\), in this case \eqref{eq-pushest2} is always satisfied while \eqref{eq-pushest} may not even make sense since \(X\) may not be differentiable.
In our case, see Proposition \ref{prop-boundedflow}, we obtain an even weaker condition,
\begin{equation}\label{eq-pushest3}
||De^{tX}_p||_{\infty} \lesssim e^{tM}
\end{equation}
where \(M\)  is \(d\eta\) evaluated at two specific directions, one in \(\ker(\eta)\) and the other in the transverse subspace of \(\ker(\eta)\), see \eqref{def-Mk}. In particular, the fact that \(d\eta\) is evaluated at a vector in the kernel of \(\eta\) plays an important role in bounding the value of \(M\) in specific  applications.

The bound \eqref{eq-pushest3} also comes into play in the proof of uniqueness under the exterior regularity condition. The key point of the proof is first of all to reduce the problem to that of uniqueness of solutions for ODE's, as we show below. To prove the uniqueness for ODE's we use an innovative argument based on Stoke's Theorem rather than the more standard approach based on Gronwall's inequality.
To present a brief conceptual overview of the argument, we consider for simplicity a vector field \(  X  \) on a surface.

For smooth vector fields we can define a change of coordinates  that straightens out the integral curves and we can  define a differential \(  1  \)-form \(  \alpha  \) with \(  X\in ker (\alpha)  \) and \(  d\alpha=0  \). Uniqueness of solutions is then an easy consequence of Stoke's Theorem:   the integration of
\(\alpha\) along any closed curve is zero and so,  by contradiction, if \(X\) is not
uniquely integrable at a point there is  a closed curve \(  \gamma  \) formed by two integral curves of \(X\) and a curve \(  \lambda  \) transversal to \(X\) (as in Figure \vref{fignonu}). The integral of \(\alpha\) along \(  \gamma  \)  is non-zero because  \(  \alpha(X)=0  \) and only one piece, \(  \lambda  \), of \(  \gamma  \) is  transverse to \(X\), and thus we get a contradiction. In the case of \emph{continuous} vector fields we consider a sequence of smooth approximations \(  X^{k}  \) of \(  X  \) and corresponding differential 1-forms \(  \alpha^{k}  \) (which do not necessarily have to converge to \(  \alpha  \)). Integrating these forms \(  \alpha^{k}  \) along the very same closed curve \(  \gamma  \) we cannot apply the exact same argument because we may have \(  \alpha^{k}(X)\neq 0  \) but, using Equation \eqref{eq-pushest3}, we can show that \(  |\alpha^{k}(X)|\to 0  \) as \(  k\to \infty  \) and we then show that this  is sufficient to obtain uniqueness for \(  X  \).

\section{Existence of Integral Manifolds}\label{subsec-existence}
In this section we are going to prove the existence of integral manifolds under the asymptotic involutivity, thus proving the first part of the Main Theorem. The general strategy is quite geometric and intuitive. We construct a sequence of local integral manifolds \( W^{k} \) and show that they converge to a manifold which is an integral manifold of the distribution \( E \). The approximating manifolds \( W^{k} \) will be constructed in terms of the approximating \( C^{1} \) distributions \( E^{k} \) but are of course in general not integral manifolds of these distributions since the \( E^{k} \) are not in general integrable. We can measure how far these manifolds are from being integral manifolds by comparing their tangent spaces to the distributions \( E^{k} \) and the key step in the proof will consist of relating this ``distance'' to the quantities involved in the definition of asymptotic involutivity in terms of the forms which define \( E^{k} \) and their derivatives.

To emphasize the generality of our approach, we work first in the context of an arbitrary \( C^{1} \) distribution.
In section \ref{sec-aim} we define a $C^{1}$ manifold $\mathcal{W}$ associated to this distribution, and state the key estimate in Proposition \ref{prop-uniformsurfaces} which bounds the ``non-integrability'' of $\mathcal{W}$. We reduce the proof of Propositon \ref{prop-uniformsurfaces} to that of a more technical Proposition \ref{prop-boundedflow} which uses the pushforward of vector-fields insides this distribution.  In Subsection \ref{ssec-vfields} we prove Proposition \ref{prop-boundedflow} and then in Subsection \ref{sec:convergence} we apply the estimates to our sequence of approximations.

\subsection{Almost integral manifolds}\label{sec-aim}
Let $\Delta$ be a $C^1$ $m$-dimensional bundle on an open set $U\subset\mathbb{R}^{n+m}$ for $m,n\geq1$. Fix
a point $x_0 \in M$.
We can choose a coordinate system $(x^1,...,x^m,y^1,...,y^n, U)$ centered at $x_0\in U$ so that
$\Delta$ is spanned by vector fields of the form
\begin{equation}\label{eq-vecfields2}
X_i:=\frac{\partial}{\partial x^i} +
\sum_{j=1}^{n}a^{i}_j\frac{\partial}{\partial y^{j}}
\end{equation}
for some $C^1$ functions $a_i^j$ for $i=1,...,m$. For later on use we also define

\begin{equation}\label{Y}
\mathcal{Y}_p:=\text{span}\left\{\frac{\partial}{\partial y^{\ell}}|_p, \ell=1,...,n \right\}
\end{equation}
One of the most useful properties of such vector-fields is that $[X_i,X_j]_p \in \mathcal{Y}_p$ for all $i,j=1,...,m$ and $p \in U$. This property will be used repeatedly all through out the paper.

Since the vector fields \(  X_{i}  \) are \(  C^{1}  \) in \(  U  \), they are uniquely integrable and we let $e^{tX_{i}}(p)$ denote the flow associated to \(  X_{i}  \) starting at the point \(  p  \).
Then, for \(  \epsilon_{1}>0  \) sufficiently small, we define the map $W: (-\epsilon_1,\epsilon_1)^m \rightarrow U$~by
\begin{equation}\label{W}
 W(t_1,...,t_m) = e^{t_mX_m} \circ \cdots \circ e^{t_1X_1}(x_0).
\end{equation}
The set
$$
\mathcal W:=W((-\epsilon_1,\epsilon_1)^m)
$$
is our candidate manifold that ''integrates'' the set of vector fields $\{X_i,i=1,...,m\}$. In general it is not an integral manifold of \(  \Delta  \).

Let $\{U_i\}_{i=1}^{\ell}$ be any open cover of $U$,
$\eta^i_1,...,\eta^i_n$ a basis of sections of $\mathcal{A}^1(\Delta)$ on $U_i$ and let $A^i$ be the section of the $F(\mathcal{A}^1(\Delta))$ on $U_i$ formed by these sections. We adopt all the notations given in section \ref{sec-mainthm} for these objects (but we drop the index $k$). We also denote by $A^{-1,i}_p$
the inverse of $A^i_p$ restricted to $\mathcal{Y}_p$

\begin{prop}\label{prop-uniformsurfaces}
For every $t=(t_1,...,t_m)\in(-\epsilon_1,\epsilon_1)^m$
and   $i=1,...,m$ we have
\begin{equation}\label{mainapprox}
\left|\frac{\partial W}{\partial  t_i}(W(t)) - X_i(W(t))\right|
\leq m\epsilon_1 \sup_{r,s,j \in \{1,...,\ell\}}||dA^r|_{\Delta}||_{\infty} ||A^{-1,s}||_{\infty}
e^{m\epsilon_1 M^j_{A}}.
\end{equation}

\end{prop}

Notice that if the distribution $\Delta$ satisfies the usual Frobenius involutivity condition then
$dA^r|_{\Delta}=0$ for all $r$ and then  Proposition
\ref{prop-uniformsurfaces} implies that $\partial W/\partial  t_i=X_i$ which implies that $\mathcal W$ is an integral
manifold of $\Delta$.  In our setting, the distributions \(  E^{k}  \) are not involutive but the weak asymptotic involutivity condition implies that they are increasingly ``almost involutive'' and thus, by Proposition \ref{prop-uniformsurfaces}, ``almost integrable''. In Section \ref{sec:convergence} we will show that this implies that we can pass to the limit and obtain an integral manifold for our initial distribution \(  E  \) of the Main Theorem.

Proposition \ref{prop-uniformsurfaces} follows from the next proposition which we prove in Section~\ref{ssec-vfields}.

\begin{prop}\label{prop-boundedflow}
Let \(  \Delta  \) be a $C^1$, rank \(  m  \) distribution on \(  U  \), \(  X_{1},..., X_{m}  \) a basis of \(  \Delta  \) of the form \eqref{eq-vecfields2} and \(  \mathcal Y  \) the complementary distribution of the form \eqref{Y}. Let $\{U_i\}_{i=1}^{q}$ be an open cover of $U$,  \(  \{\eta^i_{1},...,\eta^i_n\}  \) basis of sections of \(  \mathcal A^{1}(\Delta)   \) defined on $U_i$ and $A^i$ be the section of \(  F(\mathcal A^{1}(\Delta) )  \) on $U_i$ formed by these differential 1-forms so that they form a compatible cover. Then for all $(t_1,...,t_m)\in(-\epsilon_1,\epsilon_1)^m$ and
$Y\in \mathcal{Y}$ we have
\begin{equation}\label{eq-boundedflow}
|De^{t_mX_m}\circ...\circ De^{t_1X_1}_{x_{0}} Y|
\leq \sup_{\substack {s \in \{1,...,\ell\}}}|A^i_{x}(Y)|||A^{-j}_{x_m}||e^{m\epsilon_1 M^s_{A}}
\end{equation}
where $x_m = e^{t_mX_m}\circ...\circ e^{t_1X_1}(x_0)$ and $i,j$ are such that $x_{0} \in U_i$, $x_m \in U_j$.
\end{prop}

\begin{proof}[Proof of Proposition \ref{prop-uniformsurfaces} assuming Proposition \ref{prop-boundedflow}]

Observe first that by the chain rule, for $i=1,...,m$, we have
\begin{equation}\label{eq-pushforward}
\frac{\partial W}{\partial  t_i}=(De^{t_mX_m}\circ...\circ De^{t_{i+1}X_{i+1}})X_i.
\end{equation}
where $De^{t_iX_i}$ denotes the differential of the flow with respect to the spatial coordinates (to simplify the notation we omit the base points at which the derivatives are calculated because our estimates will be uniform in \( U \) and so the specific base points do not matter).
By a relatively standard result on   the calculus of vectors (see \cite[Chapter 2]{AgrSac04}),
 for any two vector fields $Z,X$ on $U$ we have
\begin{equation}\label{pusg}
 (De^{tZ}X)(x)-X(x)=\int_0^t(De^{sZ}[X,Z])(x)ds.
\end{equation}
Thus, for \(  t=(t_{1},...,t_{m})  \), using \eqref{eq-pushforward} and \eqref{pusg}, we have
\begin{align*}
\frac{\partial W}{\partial  t_i}(W(t)) - X_i(W(t)) &=
(De^{t_mX_m}\circ...\circ De^{t_{i+1}X_{i+1}}) X_i - X_i
\\
&=\sum_{j=i+1}^{m}De^{t_mX_m}\circ...\circ De^{t_{j+1}X_{j+1}}\left(De^{t_{j}X_{j}}X_i-X_i\right)
 \\
  &=\sum_{j=i+1}^{m}De^{t_mX_m}\circ...\circ De^{t_{j+1}X_{j+1}}\int_0^{t_{j}}De^{sX_{j}}[X_{i},X_{j}]ds\\
  &=\sum_{j=i+1}^{m}\int_0^{t_{j}}De^{t_mX_m}\circ...\circ De^{t_{j+1}X_{j+1}}De^{sX_{j}}[X_{i},X_{j}]ds
\end{align*}
Then taking norms on both sides we get
\begin{equation}\label{eq-tangentdif}
\left|\frac{\partial W}{\partial  t_i}(W(t)) - X_i(W(t))\right| \leq
 m\epsilon_1\max_{\substack{(t_m,...,t_1) \in [-\epsilon_1,\epsilon_1]^m \\s,r \in \{1,...,m\}}}|De^{t_mX^k_m}\circ...\circ De^{t_1X^k_1}[X_s,X_r]|_{\infty}
\end{equation}
Note that by the choice of $X_{i}$, the brackets $[X_s,X_r]$ lie in \(  \mathcal{Y}  \) so we can apply Proposition \ref{prop-boundedflow} with $Y$ replaced by $[X_s,X_r]$ to get
\begin{equation}
|De^{t_mX_m}\circ...\circ De^{t_1X_1}_x [X_s,X_r]|
\leq \sup_{\substack{s,r,j \in \{1,...,\ell\}}} |A^i_{x}([X_s,X_r])| \  ||A^{-r}||_{\infty}e^{m\epsilon_1 M^j_{A}}
\end{equation}
Then using Cartan's formula
$$
|A^i_{x}([X_s,X_r])|= |dA^i_x(X_s,X_r)|
$$
we get
\begin{equation}\label{eq-boundedflow2}
|De^{t_mX_m}\circ...\circ De^{t_1X_1}_x [X_s,X_r]|
\leq \sup_{\substack{s,r,j \in \{1,...,\ell\}}} ||dA^s|_{\Delta}||_{\infty} ||A^{-r}||_{\infty}e^{m\epsilon_1 M^j_{A}}.
\end{equation}
By inserting Equation \eqref{eq-boundedflow2} into Equation \eqref{eq-tangentdif} we get Proposition \ref{prop-uniformsurfaces}.
\end{proof}

\subsection{Proof of Proposition \ref{prop-boundedflow}}\label{ssec-vfields}

Proposition \ref{prop-boundedflow} is the technical heart of the proof of existence of integral manifolds  where we use the exterior derivative of the annihilator differential forms to control the push forwards of our vector-fields.
We first define a non-autonomous flow which corresponds to flowing along a direction $X_i$ and then switching
to $X_{i+1}$ and so on. Let $(t_1,...,t_m)\in(-\epsilon_1,\epsilon_1)^m$ and $T_{i}=\sum_{\ell=1}^i|t_{\ell}|$.
We define the non-autonomous piecewise smooth vector field
 $$
 X_t:=\sigma (t_{i})X_i \quad\text{ if }\quad T_{i} \leq t < T_{i+1}
 $$
where $\sigma$ is the sign function. Its associated non-autonomous flow is denoted by $\phi(t)$. With this notation we have that for any $T_{i}< t< T_{i+1}$
\[
\phi(t) = e^{\sigma(t_i)(t-T_{i})X_{i}}\circ...\circ e^{\sigma(t_{1})T_{1}X_1} (x),
\]
\[
Y_t = D\phi(t)Y = De^{\text{s}(t_i)(t-|t_i|)X_{i}}\circ...\circ De^{t_1X_1}_x Y .
\]
Let $\phi$ be the piecewise smooth curve which is the image of the map $\phi: [0,T_{m}]\rightarrow U$.
Recall now that we had a cover $\{U_{i}\}_{i=1}^{q}$ of $U$. We can take the intersection of $\phi$ with these open sets and consider the connected components of these intersections, which are curves in $\phi$, and which we denote
 by $\{I_j\}_{j=1}^{u}$. By shrinking and reindexing $I_i$ we can assume that $I_{i+1} = \phi([s_i,s_{i+1}])$ with $s_0=0$, $s_u = T_{m}$, $s_i < s_{i+1}$ and that each $I_i$ is inside one of the elements $U_{\ell_{i}}$ of the covering $\{U_i\}_{i=1}^q$. We let $\{A^{\ell_i}\}_{i=1}^u$ denote
restrictions of the sections of $F(\mathcal{A}^1(\Delta))$ defined on the $U_{\ell_i}$'s.

\begin{lem}\label{nonaut} For every $i=1,...,u$, $s_{i} > s \geq s_{i-1}$ and $Y \in \mathcal{Y}$ we have
\begin{equation}\label{eq1}
A^{\ell_i}_{s}(Y_{s})= A^{\ell_i}_{s_{i-1}}(Y_{s_{i-1}}) +\int_{s_{i-1}}^{s}dA^{\ell_i}(X_{\tau}, Y_{\tau})(\phi(\tau))d\tau
\end{equation}
\end{lem}

\begin{proof}[Proof of Proposition \ref{prop-boundedflow} assuming Lemma \ref{nonaut}]
Equation \ref{eq1} can be rewritten as
\begin{equation}
A^{\ell_i}_{s}(Y_s)= A^{\ell_i}_{s_{i-1}}(Y_{s_{i-1}}) +\int_{s_{i-1}}^{s}dA^{\ell_i}_{\tau}(X_{\tau}, A^{-1,\ell_i}_{\tau}\circ A^{\ell_i}_{\tau}Y_{\tau})(\phi(\tau))d\tau
\end{equation}
This tells us that $A^{\ell_i}_{t}(Y_t)$ is the solution of the ODE
$$
\frac{dF}{dt} = dA^{\ell_i}_t(X_t, A^{-1,\ell_i}_t\circ F_t) \quad \text{for} \quad {s}_{i-1} < t < {s}_{i}
$$
with initial condition
$$
F({s}_{i-1}) = A^{\ell_i}_{s_{i-1}}(Y_{s_{i-1}}).
$$
This ODE is linear and piecewise $C^1$ in $t$ and $C^1$ in other variables so has unique solutions. Let $G^i_t$ be the fundamental matrix of this ODE which satisfies (see \cite{Har02} for instance)
\begin{equation}\label{eq-fun1}
|G^{i}_t| \leq e^{|t-s_i| \ ||dA^{\ell_i}(X_t,A_t^{-1,\ell_i}\cdot)||_{\infty}} \leq e^{|s_{i+1}-s_i|M^{\ell_i}_{A}}.
\end{equation}
Moreover
$$
 A^{{\ell_i}}_{s_{i}}(Y_s)=  G_s \circ A^{\ell_i}_{s_{i-1}}(Y_{s_{i-1}}),
$$
and so we have
\begin{equation}\label{eq-sol1}
Y_s= A^{-1,{\ell_i}}_{s_{i}} \circ G_s \circ A^{\ell_i}_{s_{i-1}}(Y_{s_{i-1}}).
\end{equation}
So repeatedly applying \eqref{eq-sol1} and using \eqref{eq-fun1}, we get
\begin{equation}\label{eq-mainineq}
|Y_{s_u}| \leq ||A^{-1,{\ell_u}}_{{s}_{u}}|| \ |A^{{\ell_1}}_{s_0}(Y)| \  \prod_{i=2}^{u-1}  ||A^{{\ell_i}}_{{s}_{i}}\circ A^{-1,{\ell_{i-1}}}_{{s}_{i}}|| \   \prod_{i=1}^{u} ||G^{i}_{{s}_{i+1}}||.
\end{equation}
But now, by assumption of compatible cover we get $||A^{{\ell_i}}_{{s}_{i}}\circ A^{{-1,\ell_{i-1}}}_{{s}_{i}}||=1$, and by \eqref{eq-fun1} we get
$$
 \prod_{i=1}^{u} ||G^{i}_{{s}_{i+1}}|| \leq e^{m\epsilon_1M^{\ell_i}}.
$$
We remind the reader that $s_0=0$, $s_u=T_{m}$, $Y_0 = Y$, $Y_{T_{m}} = De^{t_mX_m} \circ ... \circ De^{t_1X_m}_xY$, and so from equation \eqref{eq-mainineq}
we get
$$
|De^{t_mX_m} \circ ... \circ De^{t_1X_m}_xY| \leq  \sup_{\substack {s \in \{1,...,\ell\}}}||A^{-1,{\ell_u}}_{x_m}|||A^{{\ell_1}}_{x}(Y)|e^{m\epsilon_1M^s_{A}}
$$

\end{proof}
To prove Lemma \ref{nonaut},   first note that
$$
A^{\ell_i}_{s}(Y_{s}) = (\eta^{\ell_i}_1(Y_{s}),....,\eta^{\ell_i}_n(Y_{s}))
\text{ and }
dA^{\ell_i}_s(X_s,Y_s) = (d\eta^{\ell_i}_1(X_s,Y_{s}),....,d\eta^{\ell_i}_n(X_s,Y_s)).
$$
So it is sufficient to prove \eqref{eq1}  for a fixed differential form $\eta^{\ell_i}_j$ defined on $U_{\ell_i}$. For convenience in this part we will drop the index $\ell_i$ and denote the evaluation points as subscripts. Therefore we need to prove
\begin{equation}\label{eq2}
\eta(Y_{s})_{\phi(s)}= \eta(Y_{s_{i-1}})_{\phi(s_{i-1})} +\int_{s_{i-1}}^{s}d\eta(X_{\tau}, Y_{\tau})(\phi(\tau))d\tau.
\end{equation}
 We will first consider the case when the flow $\phi(t)$ is obtained from a single vector field, that is $\phi(t) = e^{tX_i}(x)$. The general case will be deduced from this one.
\begin{lem}\label{lem-push}
For every $x\in U$, $Y \in \mathcal{Y}$ and  $|t_i|$ small enough so that $x_i=e^{t_iX_i}(x) \in U$ we have
$$
\eta_j(De^{t_iX_i}_xY)_{x_i}= \eta_j(Y)_{x} + \int_{0}^{t_i}d\eta_j(X_i, De^{sX_i}_xY)_{e^{sX_i}(x)}ds
$$
for all $i=1,...,m$, $j=1,...,n$.
\end{lem}
\begin{proof}
 Let $\gamma$ be a curve defined on $[0,\tilde{t}]$ such that $\gamma(0)=x$, $\gamma'(0)=Y$ and ${e}^{t_iX_i}(\gamma)\subset U$. Note that $X_i$ is always transverse to $\gamma$. Denote
$y=\gamma(\tilde{t})$. Define the
parameterized surface $\Gamma$ by

$$r(s_1,s_2)= {e}^{s_2X_j} \circ \gamma(s_1)$$
for $0<s_1\leq\tilde{t}$ and $0< s_2\leq t_i$. Then the boundary of $\Gamma$ is composed of the curve $\gamma$ and the following
piecewise
smooth curves (see Figure \ref{figstoke}):
$$\xi_1(s) ={e}^{sX_i}(x) \quad \quad \xi_2(s) = {e}^{sX_i}(y) \quad \quad \beta(s) = {e}^{t_iX_i}\circ \gamma(s).
$$
\begin{figure}[h]
  \centering
  \includegraphics[width=80mm]{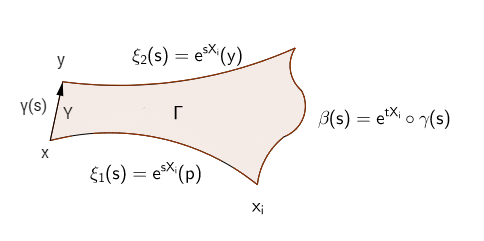}
  \caption{Applying Stoke's Theorem}\label{figstoke}
\end{figure}
Since $\eta_j(X_i)= 0$ for all $i,j$, using Stoke's theorem we have
$$\int_{\beta}\eta_j - \int_{\gamma}\eta_j = \int_{\Gamma} d\eta_j$$
which gives
\begin{equation}\label{eq-equality}
\begin{aligned}
\int_0^{\tilde{t}}\eta_j(\beta'(s_1))ds_1 = \int_{0}^{\tilde{t}}\eta_j(\gamma'(s_1))ds_1 &+\\
\int_{0}^{\tilde{t}}& \int_{0}^{t_i}d\eta_j(\frac{\partial r}{\partial s_1},\frac{\partial r}{\partial s_2}){(s_1,s_2)}ds_1ds_2.\\
\end{aligned}
\end{equation}

Differentiating \eqref{eq-equality} with respect to $\tilde{t}$ at $\tilde{t}=0$  we have

$$\eta_j(\beta'(0)) =\eta_j(\gamma'(0)) + \int_{0}^{t_i} d\eta_j(\frac{\partial r}{\partial s_1},\frac{\partial r}{\partial
s_2}){(0,s_2)}ds_2.$$
Using chain rule we have

$$\beta'(0)=D{e}^{t_iX_i}Y|_x \quad \quad \gamma'(0)=Y$$
and

$$\frac{\partial r}{\partial s_1}(0,s_2) = D{e}^{s_2X_i}Y|_x \quad \quad \frac{\partial r}{\partial s_2}(0,s_2) =
X_{i}({e}^{s_2X_i}(x))$$
one can write the equality \eqref{eq-equality} as

$$\eta_j(D{e}^{t_iX_i}_xY_{x_i}) =\eta_j(Y)_{x} + \int_{0}^{t_i}d\eta_j(D{e}^{s_2X_i}_xY,X_{s_2})_{r(0,s_2)}ds_2$$
which concludes the proof of the lemma.
\end{proof}
Our next step is to generalize Lemma \ref{lem-push} for the composition of differentials.

\begin{lem}\label{lem-comppush}
 For every $(t_1,...,t_m)\in(-\epsilon_1,\epsilon_1)^m$, $Y \in \mathcal{Y}$ and $j=1,...,n$ we have
$$\eta_j(De^{t_mX_m}\circ\cdots\circ De^{t_1X_1}Y)_{x_m}= \eta_j(Y)_{x_{0}}
 +\sum_{i=1}^m\int_{0}^{t_i}d\eta_j(X_i, De^{sX_i}\circ\cdots\circ
De^{t_1X_1}Y)(x_i(s))ds $$
 where $x_m=e^{t_mX_m}\circ\cdots\circ e^{t_1X_1}(x_0)$ and $x_i(s)=e^{sX_i}\circ\cdots\circ e^{t_1X_1}(x_0).$
\end{lem}
Lemma \ref{lem-comppush} will follow by successive applications of Lemma \ref{lem-push}. However we need to check first that the pushforward of the flows leaves the $ \mathcal{Y}$ subspace invariant.
\begin{lem}\label{lem-inv}
 For every $i=1,...,m$, $p\in U$,  $|t_i|<\epsilon_1$ and \(Y \in \mathcal{Y}\) we have
 $$ De^{t_iX_i}_pY\in \mathcal{Y}.$$
\end{lem}
\begin{proof}
Let $p=(x^1_0,...,x^m_0,y^1_0,...,y^n_0)$. Recall that $X_i = \frac{\partial}{\partial x^i} + \sum_{j=1}^n a^{ij}(x,y)\frac{\partial}{\partial y^j}$ and so the flow of this vector field is
$e^{tX_i}(x^1_0,...,x^m_0,y^1_0,...,y^n_0) = (x^1_0,...,x^i_0+ t,...,x^m_0, y^1(t,x_0,y_0)$ $,...,y^n(t,x_0,y_0))$ where $y^i(t,x,y)$ are functions
$C^1$ in their variables. Since the differential has the form
$$D(e^{tX_i})|_{(x,y)}=\left(
  \begin{array}{cccccc}
     Id_{m\times m}   &  0_{m\times n}   \\
     A_{n\times m}   &  B_{n\times n}   \\
  \end{array}
\right)
$$
for some matrices $A$ and $B$, the invariance of vectors in $\mathcal{Y}$ follows directly.

\end{proof}

\begin{proof}[Proof of Lemma \ref{lem-comppush}]
Let $(t_1,...,t_m)\in(-\epsilon_1,\epsilon_1)^m$ and $Y \in \mathcal{Y}$,
we first carry out the proof for $t_1,t_2 \geq 0$, $t_j=0$ for $j>2$.  First note that by Lemma \ref{lem-inv}
$De^{t_1X_1}_pY  \in \mathcal{Y}$. So applying Lemma \ref{lem-push} twice, one has that
$$\eta_j(De^{t_2X_2} \circ De^{t_1X_1}Y)_{x_2}= \eta_j(Y)_{x_2} + \int_{0}^{t_1} d\eta_j(X_1, De^{sX_1}Y)(x_1(s))ds$$
$$+\int_{0}^{t_2}d\eta_j(X_2, De^{sX_2}\circ De^{t_1X_1}_p Y)(x_2(s))ds.$$
The general case follows in the same way, by applying Lemma \ref{lem-push} repeatedly.
\end{proof}

\subsection{Convergence to Integral Manifolds}\label{sec:convergence}
In this section we show that an asymptotic involutive bundle is integrable, thus proving the first part of the Main Theorem. We suppose that the asymptotic involutivity in Definition \ref{defn-asyminv2} is satisfied, in particular we have the sequences of
differential forms
$\{\eta^{k,i}_1,..., \eta^{k,i}_n\}$ defined on open sets $U^k_i$ of the covering of $U$
and the sequence of bundles $E^{k}$ defined on $U$ which converges to
the continuous bundle
$E$. As before we can choose a coordinate system $(x,y)$ independent of $k$ where $E^{k}$ and $E$ are spanned by vector fields of the form \eqref{eq-vecfields2} (though for \(  E  \) the vector fields are only continuous). We recall that
  $A^{-k,i}_p$ denotes the inverse of $A^{k,i}_p$ restricted to $\mathcal{Y}_p$.
For $k>1$, let $W^k$ be the analogous of the map defined in \eqref{W} for $\Delta$.
By Proposition \ref{prop-uniformsurfaces}, for every $k>1$ we have
\[
\left|\frac{\partial W^k}{\partial  t_i}(W^k(t)) - X^k_i(W(t))\right|\leq   m\epsilon_1sup_{\substack{s,r,j \in \{1,...,\ell\}}}||dA^{k,s}|_{\Delta}||_{\infty}||A^{-k,r}||_{\infty}e^{m\epsilon_1M^{k,j}_{A}}
\]
Choosing $\epsilon_1$ small enough so that $\epsilon_1m \leq \epsilon$ and  using asymptotic involutivity we have
$$
\lim_{k\to0}\left|\frac{\partial W^k}{\partial  t_i}(W^k(t)) - X^k_i(W(t))\right|=0.
$$
In particular we have uniformly sized manifolds $\mathcal{W}^k$ whose tangent spaces converge to \(  E  \) in angle as $k \rightarrow \infty$.  This is enough to show that these manifolds converge to some manifold \(  \mathcal W  \) and that this is an integral manifold of \(  E  \). This fact is quite intuitive but we give a proof of such a statement in a more abstract setting for completeness.

\begin{prop}\label{prop-convergence} Let $E$ be a continuous tangent subbundle of rank $m$ defined on $U$ and $E^k$ be a $C^1$ approximation of $E$. Let $\mathcal{V}^k \subset U$ be a sequence of $C^1$ manifolds of dimension $m$ and of uniform size with a point $p$ in common.
Assume that

\begin{equation}\label{eq-convergence}\lim_{k\rightarrow 0}\sup_{q \in \mathcal{V}^k}\angle (T_q\mathcal{V}^k,E^k_q)=0\end{equation}
Then there exists a subsequence of submanifolds $\mathcal{W}^k \subset \mathcal{V}^k$, which converges to an $m$-dimensional manifold $\mathcal{W}$, which is an integral manifold of $E$ passing through $p$.
\end{prop}

\begin{proof} We choose coordinates $(x^1,...,x^m,y^1,...,y^n)$ so that $E_p= \text{span}\{\frac{\partial}{\partial x^i}\}_{i=1}^m$ and denote $\mathcal{Y} = \text{span}\{\frac{\partial}{\partial y^i}\}_{i=1}^n$. We shrink $U$ if necessary so that each $E_q$ is transverse to $\mathcal{Y}_q$ for all $q \in U$. Since $E^k$ converges to $E$,

$$\lim_{k\rightarrow 0}\sup_{q \in \mathcal{V}^k}\angle (T_q\mathcal{V}^k,E_q)=0.$$
Along with this, one has that $\mathcal{V}^k$ have uniform size so for $k$ large enough we have submanifolds $\mathcal{W}^k \subset \mathcal{V}^k$ which can be written as graphs of functions $G^k: V\subset E_p \rightarrow U$. Thus we can write $\mathcal{W}^k$ as the images of the functions:

$$W^k(x^1,...,x^m) = (x^1,...,x^m,G^k(x^1,...,x^m))$$
where

$$X^k_i=\frac{\partial W^k}{\partial x^i} = \frac{\partial}{\partial x^i} + \sum_{j=1}^m\frac{\partial G^k_j}{\partial x^i}\frac{\partial}{\partial y^j} $$
span the tangent space of $\mathcal{W}^k$. Note first that the differential $DW^k$ has $X^k_i$ as its columns which are linearly independent so $\mathcal{W}^k$ are $C^1$ embedded manifolds. Since the tangent space of $\mathcal{W}^k$ converges in angle to $E$, which is transverse to $\mathcal{Y}$, one has that
$\sup_{k,i}|X^k_i|_{\infty} \leq C_1$ for some constant $C_1>0$. Since the differential $DW^k$ is a matrix whose columns are $X^k_i$, we get that
$\sup_{k,i}|DW^k_i|_{\infty} \leq C_2$ for some constant $C_2>0$. Therefore the sequence of functions $W^k_i$ is equi-Lipschitz and equi-bounded and so up to choosing a subsequence, converges to a continuous function $W: V\rightarrow U$.

We are left to prove that \(\mathcal{W}:=W(V)\)  is a \(C^{1}\) \(m\)-dimensional manifold tangent to \(E\).  For this it is sufficient to prove that \(X^k_i\)'s converges to some linearly independent \(X_{i}\)'s that span \(E\).
This implies that $DW^k$ is a matrix which converges to another matrix, say $A$ whose columns are $X_i$.
Thus we have that $p \in \mathcal{W}^k$ for all $k$, $W^k \rightarrow W$ and $DW^k \rightarrow A$. Therefore in fact $W: V\rightarrow U$ is a $C^1$ function whose derivative is a matrix whose columns are $X_i$. Therefore $\mathcal{W}$ is a $C^1$ manifold that is tangent to $E$ and passes through $p$.

To prove the convergence of \(X^{k}_{i}\)'s, we first observe that $\text{span}\{X^k_i\}_{i=1}^m$ at $q_{k}=W^k(x^1,...,x^m)$ converges to $E$ at $q=W(x^1,...,x^m)$. Moreover $E$ can be spanned by vector fields of the form
$$X_i= \frac{\partial}{\partial x^i} + \sum_{j=1}^mH_{ij}(x,y)\frac{\partial}{\partial y^j}.$$
Since the tangent space of $\mathcal{W}^k$ converges to $E$, there exist vector fields $Y^k_i$ inside the tangent space that converges to each $X_i$.
But $Y^k_i = \sum_{j=1}^n a^k_{ij} X^k_j$ and by the form of $X_i$ we see that $a^k_{ii}\rightarrow 1$ while $a^k_{ij}\rightarrow 0$
for $j \neq i$ as $k\rightarrow \infty$. And so in particular $|X^k_i - Y^k_i|_{\infty} \rightarrow 0$ as $k\rightarrow \infty$ which implies that $X^k _i$ converges to $X_i$.
\end{proof}

\section{Uniqueness of Local Integral Manifolds}\label{subsec-uniqueness}

In the previous section we have proven that asymptotic involutivity implies integrability of $E$. In this section we
assume that $E$ is integrable and show that if
$E$ is exterior regular then it is uniquely integrable which will
then conclude the proof of the Main Theorem.
We remind that uniquely integrable means the following:
When ever two integral manifolds of $E$ intersect, they intersect in a relatively open (in both) set. First we reduce the question of unique integrability to the following:


\begin{prop}\label{thm-Hartmanunique}
Assume $E$ is integrable and  exterior regular then $E$ is spanned by a
linearly independent set of vector fields $X_i$ which are uniquely integrable.
\end{prop}



This immediately implies uniqueness for the integral manifolds.
\begin{proof}[Proof of Uniqueness part of Main Theorem assuming Proposition \ref{thm-Hartmanunique}]
Assume that there exist two integral manifolds $\mathcal{W}_1,\mathcal{W}_2$ of $E$ such that $z \in \mathcal{W}_1 \cap
\mathcal{W}_2$.
By assumption we have that $E$ is spanned by
uniquely
integrable vector fields $X_i$. Moreover since $\mathcal{W}_j$ for $j=1,2$ are integral manifolds of $E$, for any $q_j \in \mathcal{W}_j$
$X_i(q_j) \in T_{q_j}\mathcal{W}_j$ for all $i$ and $j$. Now for
$\epsilon$  small enough the $m$-dimensional surface \(\mathcal{W}=\{e^{t_mX_m}\circ ...\circ e^{t_1X_1}(z): |t_i| \leq \epsilon\}$
is well defined.
Moreover by unique integrability of $X_i$ restricted to each $\mathcal{W}_j$, \(\mathcal{W}\) is a subset of both surfaces. This means that the
intersection of $\mathcal{W}_1$ with $\mathcal{W}_2$ is relatively open in both surfaces.
\end{proof}

To prove Proposition \ref{thm-Hartmanunique} we define quite explicitly the linearly independent vector fields \(  X_{i}  \) which span \(  E  \) and show that they are uniquely integrable. We first introduce some notation which we will need for the argument.  Note that the assumption of exterior regularity means that there exists a sequence of approximations $E^k$ of $E$ defined on \(U\) and for each $k>0$,
there exists a covering $\{U^k_i\}_{i=1}^{n_k}$ of $U$,
a basis of sections $\{\beta^{k,i}_j\}_{j=1}^n$ of $\mathcal{A}^1(E^k)$
defined on each $U^k_i$ and the section $B^{k,i}$ of $F(\mathcal{A}^1(E^k))$ formed by these differential 1-forms.
As in the previous sections we can find $\{X^k_i\}_{i=1}^m$ a basis of vector
fields which span $E^k$ such that they converge to $\{X_i\}_{i=1}^m$ which is a basis of vector fields for $E$ such that they have the form
$$
X^k_i= \frac{\partial}{\partial x^i} + \sum_{j=1}^{n}b^{i,k}_j\frac{\partial}{\partial y^{j}}
\quad\text{ and } \quad
X_i= \frac{\partial}{\partial x^i} + \sum_{j=1}^{n}b^{i}_j\frac{\partial}{\partial y^{j}}.
$$
Note that these approximations $E^k$ and $X^k_i$ may be different from the ones used in the previous sections.
Let $i\in\{1,...,m\}$ and $x_0 \in U$ and consider integral curves of $X_i$ passing
through $x_0=(x^1_0,...,x^m_0,$ $y^1_0,...,y^n_0)$. Due to the form of $X_i$, any integral curve $\gamma(t)$ can be written as
$$
\gamma(t) =  (x^1_0,...,x^i_0+t,...,x^m_0, y^1(t),...,y^n(t))
$$
where $y^j(t)$ are differentiable functions in $t$. In particular if an integral curve passes through the point $x_0$
then it necessarily always remains inside the $n+1$ dimensional plane  $P_i=\{x^j = x^j_0 \ \text{for} \ j \neq i\}, x_0 \in P_i$ passing
through $x_0$. Therefore it is sufficient just to prove uniqueness restricting to each such subspace. So given such an $x_0$ and $P_i$ we restrict $X^k_i, X_i, \{\beta^{k,i}_j\}_{j=1}^n, U^{k,i}, B^{k,i}$ to $V_i = U \cap P_i$ with coordinates
$(x^i,y^1,...,y^n)$. For simplicity we will omit the index $i$. Note that $\{\beta^{k}_j\}_{j=1}^{n}$ are all non-vanishing and linearly independent and $X^k_i$ is in the kernel of $\beta^{k}_j$'s.  Moreover $B^{k}$ and $X^k$ restricted to these subspaces still satisfy the exterior regularity conditions.

\subsection{A Condition for Unique Integrability of \(  X  \)}
We can now start the proof of the proposition. By contradiction, we assume $X$ admits two integral curves $\gamma_1(t),
\gamma_2(t) \subset U$ with $0\leq t \leq t_1$ which intersect but whose intersection is not relatively open. Under this assumption, without
loss of generality
we can assume that $\gamma_1(0)=\gamma_2(0)=x_0$ for some $x_0\in V_i$ and that $\gamma_1(t)\neq \gamma_2(t)$ for $0 < t \leq
t_1$. Notice that for $0\leq t\leq t_1$, $\gamma_1(t)$ and $\gamma_2(t)$ have the form
$$
\gamma_1(t) =  (t, y^1_1(t),...,y^n_1(t))
\quad \text{ and } \quad
\gamma_2(t) =  (t, y^1_2(t),...,y^n_2(t))
$$
and so in particular the end points $\gamma_1(t_1)$, $\gamma_2(t_1)$ have the same $x$ coordinate. Therefore
they can be connected to each other by a straight line segment of the form $\lambda(t) =\gamma_1(t_1)+vt$ which lies inside the plane $\mathcal{Y}$ that passes
through $(t_1,0,...,0)$. Here $v$ is the unit vector in the direction $(0,y^1_2(t_1)-y^1_1(t_1),...,y^n_2(t_1)-y^n_1(t_1))$. Let $\ell(\cdot)$ denote the length. We will show that $\ell(\lambda)>0$ leads to a contradiction thus proving the proposition.

We first prove a lemma which gives sufficient conditions, in terms of the existence of a family of differential forms, to ensure $\ell(\lambda)=0$. Then in the following sections we will show that such a family can be constructed.

\begin{lem}\label{lem-difprop}
    Assume $\{\alpha^k\}_{k}$ is a sequence of $C^1$ differential forms  defined on some domain $ \tilde{V}\subset
V$ containing the curves $\gamma_1,\gamma_2,\lambda$, a constant $c>0$ such that for every $k$:
\begin{align}
&1. \ \alpha^k(X^k)=0\label{ker}\\
&2. \ d\alpha^k=0 \label{eq-dif1} \\
&3. \ \min_t{\alpha^k(\dot{\lambda}(t))}\geq c  \label{eq-dif2}  \\
&4. \ \lim_{k \rightarrow \infty}|\alpha^k(X^k-X)|_{\infty}=0 \label{eq-dif3}
\end{align}
Then $\ell(\lambda)=0$.
\end{lem}

\begin{center}
\begin{figure}
  \centering
  \includegraphics[width=40mm]{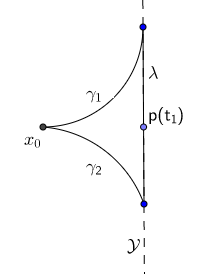}
  \caption{}\label{fignonu}
\end{figure}
\end{center}

\begin{proof}
Let $S$ be a surface in $\tilde{V}$ bounded by the curves $\gamma_1,\gamma_2, \lambda$ (whose union forms a simple, closed, piecewise smooth curve, see Figure \ref{fignonu}). By Stoke's Theorem we have
$$
\int_{\gamma_1}\alpha^k + \int_{\lambda}\alpha^k - \int_{\gamma_2}\alpha^k = \int_{S}d\alpha^k
$$
and so
$$
|\int_{\lambda}\alpha^k|\leq |\int_{\gamma_2}\alpha^k -   \int_{\gamma_1}\alpha^k + \int_{S}d\alpha^k|.
$$
Since $|\int_{\lambda}\alpha^k| \geq \min_t{\alpha^k(\dot{\lambda}(t))}\ell(\lambda)$ we can write this as
\begin{equation}\label{eq-stokesin}\min_t{\alpha^k(\dot{\lambda}(t))}\ell(\lambda) \leq |\int_{\gamma_2}\alpha^k|  +
|\int_{\gamma_1}\alpha^k| + |\int_{S}d\alpha^k|\end{equation}
Using Equations \eqref{eq-dif1}, \eqref{eq-dif2} and \eqref{eq-stokesin} we have

\begin{equation}\label{eq-lem2}
\ell(\lambda) \leq \frac{1}{c}(|\int_{\gamma_2}\alpha^k|  + |
\int_{\gamma_1}\alpha^k|)
\end{equation}
and using Equation \eqref{ker} and $\dot{\gamma_{1}}(t) = X(\gamma_1(t))$ we get
$$
 |\int_{\gamma_1}\alpha^k|  = |\int_{0}^{t_1}\alpha^k(X)(\gamma_1(s))ds| = |\int_{0}^{t_1}\alpha^k(X^k-X)(\gamma_1(s))ds|
 $$
which implies
\begin{equation}\label{eq-lem1}|\int_{\gamma_1}\alpha^k|  \leq 2t_1|\alpha^k(X^k-X)|_{\infty}.\end{equation}
The same applies to $\gamma_2$. Then plugging Equation \eqref{eq-lem1} into \eqref{eq-lem2}, we get that for all~$k$
$$
\ell(\lambda)\leq\frac{1}{c}|\alpha^k(X^k-X)|_{\infty}
$$
which, due to  \eqref{eq-dif3}, goes to $0$ as $k$ goes to $\infty$.
\end{proof}

\subsection{Definition of $\alpha^k$}

To construct $\alpha^k$ satisfying conditions of Lemma \ref{lem-difprop}, we are going to define a change of coordinates that straightens flow of each $X^k$. Since the image of the flow is a straight lines of the form $\frac{\partial}{\partial t}$, then it will also be nullified by constant differential forms of the form $dz^j$. Pulling back these constant differential forms will give us the required differential forms $\alpha^k$(see Figure \ref{figst}).  Now we make this more precise.

Let $\epsilon_1>0$ be small enough such that the box $U_1=(-\epsilon_1,\epsilon_1)^{n+1}$ centered at $0$ is in $V$. Let $\epsilon_2
< \epsilon_1$ be small enough so that for $U_2=(-\epsilon_2,\epsilon_2)^{n+1}$, $e^{tX^k}(U_2) \subset U_1$ for all $|t|\leq \epsilon_2$.
We decrease $t_1$ if necessary so that $\gamma_{\ell}$ for $\ell=1,2$ are in $U_{2}$ and $t_1 < \epsilon_2$ and denote the line
$p(t) =  (t,0,
0,...,0)$. Given some $t \leq \epsilon_2$, we denote the subspaces:
$$
P_t=\{x=t\} \cap U_{2} \simeq [{\epsilon_2},{\epsilon_2}]^n
\quad\text{ and }\quad
D^k_{\epsilon_2} = \bigcup_{0\leq t< \epsilon_2}e^{tX^k}(P_{t_1}).
$$
The domains  $D^k_{\epsilon_2}$ will be the domains on which we will define the forms $\alpha^k$. The assumptions of
of Lemma \ref{lem-difprop} require that they should contain the curves $\gamma_{1},\gamma_{2}, \lambda$. This is proven in the next lemma.

\begin{lem} There exists an open subset $\tilde{V} \subset \cap_{k=1}^{\infty} D^k_{\epsilon_2}$ such that for $t_1$ small enough, it contains the curves
$\gamma_1,\gamma_2$ and $\lambda$.
\end{lem}

\begin{proof}Fix $\delta \leq {\epsilon_{2}}/{2}$. Choose $t_1$ small enough so that for $\ell=1,2$,
$$
d(\gamma_{\ell}(t),p(t)) = |(0,y_{\ell}^{1}(t),...,y^n_{\ell}(t))| \leq \delta
$$
for all $0 \leq t \leq t_1$ and that $e^{sX^k}(P_{t_1})\cap  V_{2}$  contains a box $[-\frac{\epsilon_2}{2},\frac{\epsilon_2}{2}]^n \subset
P_{t_1+s}$ centered at $p(t_1+s)$ for all $|s| \leq t_1$. These conditions are possible to obtain since the $X^k$ are uniformly bounded in norm which guarantees that $\lambda$ and $\gamma_{\ell}([0,t_1])$ are in $D^k_{\epsilon_2}$ and that $D^k_{\epsilon_2}$ all contain a uniformly sized
box centered at the axis $y=0$ (see Figure~\ref{figdomains}).
\end{proof}

\begin{figure}[h]
  \centering
  \includegraphics[width=70mm]{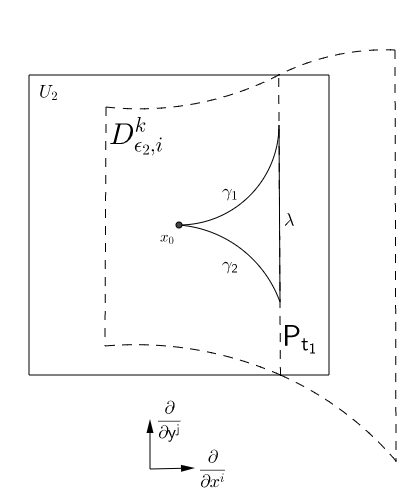}
  \caption{The domains}\label{figdomains}
\end{figure}

Note that each $P_t$ is a codimension one subspace of $U_2$. Since $P_{t_1}$ is transverse to $X^{k}$, $D^k_{\epsilon_2}$ is an $n+1$
dimensional open subset of $U_1$. Let $\phi: [-\epsilon_2,\epsilon_2]^n \rightarrow P_{t_1}$ be a parametrization of $P_{t_1}$ with coordinate representation
$\phi(z^1,...,z^n)$.
We can assume that $D\phi \frac{\partial}{\partial z^i} = \frac{\partial}{\partial y^i}$ for all $i=1,...,n$.
Then we define the change
of coordinates :
$$
\psi_{k}: [-\epsilon_2,\epsilon_2]^{n+1} \rightarrow D^k_{\epsilon_2}  \subset U_1
$$
by
$$
\psi_{k}(t,z^1,...,z^m)= e^{tX^k}(\phi(z^1,...,z^m)).
$$
This simply takes points of $P_{t_1}$ and flows them by an amount equal to $t$.

\begin{lem}
The maps $\psi_{k}$ are diffeomorphisms onto their image.
\end{lem}
\begin{proof} We will show that these maps are diffeomorphisms by showing that they are local diffeomorphisms and that they are injective.
To show that it is a local diffeomorphism, it is enough to show that the columns of $D\psi_{k}$ are everywhere linearly independent.
One of the columns is:
$$
\frac{\partial \psi_{k}}{\partial t} = X^k
$$
while the others are of the form
$$
\frac{\partial \psi_{k}}{\partial \tilde{y}^i} = De^{tX^k}\frac{\partial}{\partial y^i}.
$$
By Lemma  \ref{lem-inv}, for  $t$ small enough, $De^{tX^k}$ preserves $\mathcal{Y}$ and $\mathcal{Y}$ is transverse to $X^k$. Therefore
$D\psi^k$ is invertible at every point and therefore is a local diffeomorphism. So it remains to show it is injective. If it is not injective
then there exists two integral curves $\xi_1, \xi_2$ that start at $P_{t_1}$ and intersect at their final point. By uniqueness of solutions,
this means that $\xi_1 \circ \xi_2^{-1}$ is an integral curve of $X^k$ that starts at $P_{t_1}$ and comes back to $P_{t_1}$ (see Figure
\ref{figinj}). This either means that the $x$ component of $\xi_1 \circ \xi_2^{-1}$ first increases and then decreases or first decreases then
increases. Neither is possible due to the form of the vector fields $X^k = \frac{\partial}{\partial x} + ...$ which implies that the $x$ component is monotone.
\end{proof}

\begin{figure}[h]
  \centering
  \includegraphics[width=60mm]{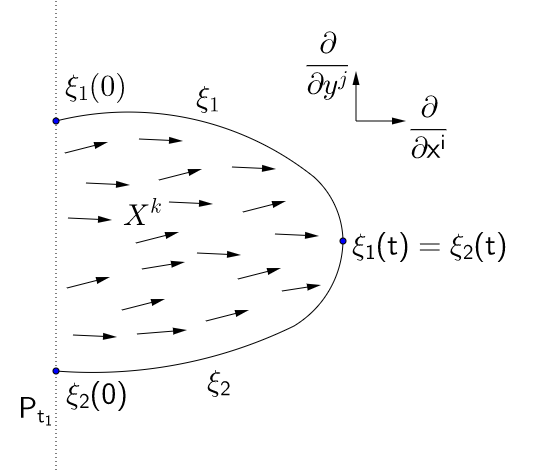}
  \caption{Injectivity}\label{figinj}
\end{figure}

Now we are ready to define $\alpha_{k}$. First, for every
 \(j=1,...,n\), let
$$
\alpha^k_j = (\psi^{-1}_{k})^*dz^j.
$$
In the next subsection, we will show that for some fixed $i_{0}$, the differential forms $\{\alpha^k_{i_{0}}\}$ are the required differential forms.

\begin{figure}[h]
  \centering
  \includegraphics[width=120mm]{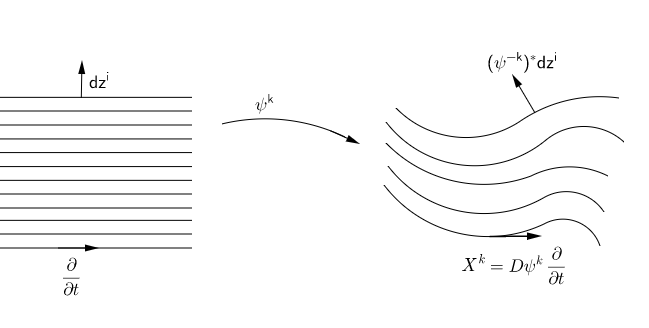}
  \caption{}\label{figst}
\end{figure}

\subsection{Choosing \(\alpha^k_{i_{0}}\)}

\begin{lem}\label{lem-renormal}
For some choice of $i_0$, the differential 1-forms $\alpha^k_{i_0}$ satisfy the conditions given in
Lemma \ref{lem-difprop}. \end{lem}

\begin{proof} First of all
$$
\alpha^k_j(X^k) = (\psi^{-1}_{k})^*dz^j(X^k) = dz^j(D\psi^{-1}_{k}X^k) = dz^j(\frac{\partial}{\partial t})=0
$$
and
$$
d\alpha^k_j= d(\psi^{-1}_{k})^*dz^j = (\psi^{-1}_{k})^*ddz^j=0
$$
which prove that conditions \eqref{ker} and
\eqref{eq-dif1} hold for all $j$ and $k$.

To check the remaining conditions, we need to calculate the inverse of $\psi_{k}$ explicitly.
One can by direct calculation check that
the inverse of $\psi_{k}$ is given by

$$\psi^{-1}_{k}(x^i,y^1,...,y^m)\\
=(x-t_1,  \phi^{-1}\circ e^{-(x-t_1)X^k}(x,y^1,...,y^n))\\
$$
Therefore $\psi^{-1}_{k}(x,y^1,...,y^m)$ is like flowing the $y$ coordinates of $(x,y^1,...,y^n)$ by the amount $-(x-t_1)$ and
replacing the first coordinate by the amount of time required to get there from $P_{t_1}$. For simplicity we denote $T=\phi^{-1}: V_i \rightarrow [-\epsilon_2,\epsilon_2]^n$, $s=-(x-t_1)$
and $(x,y^1,...,y^n)=(x,y)$.

To prove \eqref{eq-dif2}, note that $(\psi^{-1}_{k})^*$ in coordinates is the $(n+1)\times (n+1)$ matrix which transpose of the differential $D(\psi^{-1}_k)$. One has

\begin{equation}\label{eq-matrix}
D(\psi^{-1}_{k})|_{(x,y)}= \left(
  \begin{array}{ccccc}
     0   &  .  & . & 0 \\
     &   &    &   &   \\
     &  &  &   &   \\
    &  [DT \ \circ &   De^{sX^k}(x,y)]_{n\times n+1}  & \\
     &   &    &   &   \\

  \end{array}
\right)
+
\left(
  \begin{array}{ccccc}
    1  & 0 & . & . & 0 \\
     \ulcorner\urcorner & 0 & 0 & . & . \\
      & 0 & 0 & 0 & . \\
    -DT(X^k) & . & . & . & .\\
     \llcorner\lrcorner & 0 &. & . & 0\\
  \end{array}
\right)\end{equation}
For $x=t_{1}$ we have $s=0$ and by the property $De^{0X^k} = Id$, we get
\begin{equation}\label{eq-formt1}
(\psi^{-1}_{k})^*|_{x=t_1} (dz^j) = dy^j - [DT(X^k)]_jdx^j.
\end{equation}
We will now show that for some $i_0$ all $\alpha^k_{i_0}$ satisfy Equation
\eqref{eq-dif2} (at least up to changing some
orientations), i.e we will show that for some $i_0$, we have
$\alpha^k_{i_0}(\dot{\lambda})(\lambda(s))>c>0$ for all $s$. To show this, tote that the curve $\lambda$ is of the form
$\lambda(s) =sv$ for a fixed unit vector $v$ that lies inside $\mathcal{Y}$ and therefore since
$\dot\lambda(s) = v = \sum_{j=1}^n v^j\frac{\partial}{\partial y^j}$ and $\lambda \subset P_{t_1}$,
we have by Equation \eqref{eq-formt1} for all $k$
$$
\alpha^k_j(\dot\lambda(s)) =  (\psi^{-1}_{k})^*|_{x=t_1} (dy^j) (\dot\lambda(s)) =v^j.
$$
Since $|v|=1$, there exists a constant $c>0$ and $i_0$ such that $|v^{i_0}|>c$. By reversing the orientation of the loop formed by $\gamma_1,\gamma_2,\lambda$ if necessary and
therefore
reversing the direction of $\lambda$, we can assume $v^{i_0}$ is positive, that is
$$
\alpha^k_{i_0}(\dot{\lambda}(s))>c
$$
for all $k$, which proves \eqref{eq-dif2}.

Finally to prove \eqref{eq-dif3}, we will relate the quantity $|\alpha^k(X^k-X)|_{\infty}$ to the quantity given in the definition
of exterior regularity (see Definition \ref{defn-extreg2}), which goes to $0$ by assumption. Note that
$\alpha^k = (\psi^{-1}_k)^*dz^{i_{0}}$. Therefore
$$
\alpha^k(X^k-X) = dz^{i_{0}}(D\psi^{-1}_k(X^k-X))
$$
But $X^k-X = \sum_{i=1}^n(b^{i,k}-b^i)\frac{\partial}{\partial y^i} \in \mathcal{Y}$ and looking at the form of $D\psi^{-1}_k$ given in equation
\eqref{eq-matrix} we see that
$$
\|D\psi^{-1}_k|_{(x,y)}(X^k-X)\|=\|De^{sX^k}(X^k-X)\|
$$
where $s= -(x-t_1)$. By Proposition \ref{prop-boundedflow}, denoting $y=e^{sX^k}(x)$ we have
$$
\begin{aligned}
|De^{sX^k}_x(X^k-X)|_{\infty} &\leq \sup_{\substack{s,r,j \in \{1,...,s_k\}}} |B^{k,s}(X^k-X)|_{\infty}||B^{-k,r}||_{\infty}e^{m\epsilon_1M^{k,j}_{B}}\\
&\leq \sup_{\substack{s,r,j \in \{1,...,s_k\}}}\|B^{k,s}|_{E}\|\|B^{-k,r}\|_{\infty}e^{m\epsilon_1M^{k,j}_{B}}
\end{aligned}
$$
where the last inequality is given by the fact that \(  X^{k}  \) annihilates \(  B^{k,i}  \).
This quantity goes to $0$ by the exterior regularity condition which gives
\eqref{eq-dif3}.
\end{proof}

\section{Applications}\label{sec-ex}

In this section we will prove Theorems \ref{thm-ode} to \ref{thm-dyn}. Theorems \ref{thm-main} and \ref{thm-dyn}
are direct applications of our Main Theorem and Theorem \ref{thm-main} implies Theorem \ref{thm-alt2}  which implies Theorem \ref{thm-pde} which implies Theorem \ref{thm-ode}.
\subsection{Proof of Theorem \ref{thm-dyn} }\label{sec-dyn}

We first start by recalling some standard properties of dominated splittings, for details one can consult
the book \cite{Pes} and the article \cite{LuzTurWar15b}.

Let $E^0$ be a $C^1$ subbundle transverse to \(F\), then by the domination of the splitting the sequence of subbundles $E^{k}:=\phi^{-k}_*E^0= D\phi^{-k}E^0$  converges in angle to $E$ as \(k\to\infty\).
Now fix any $p \in M$ and a neighborhood $U$ of $p$, we suppose that coordinates systems \((x^1,...,x^m,y^1,...,y^n)\) are defined in \(U\) and
 all the other notations in Section \ref{sec-mainthm} are also adapted to this present Section relatively to the sequence of distributions \(\{E^k, k\geq1\}\)
 and its limit \(E\).


Let $\{V_j\}_{j=1}^N$ be a cover of $M$ by  open balls such that for each \(j\in\{1,...,N\}\), $\mathcal{A}^1(E^0)$ admits an orthonormal frame $C^{j}$.
Notice that for each \(k>1\), $\{\phi^{-k}(V_j)\}_{j=1}^l$ is an open cover of \(M\) and \(\{C^{k,j} = (\phi^k)^*C^{j}\}_{j=1}^l\) is a frame of  $\mathcal{A}^1(E^k)$
such that for  \(j=1,...,l\), \(C^{k,j}\) is defined in \(\phi^{-k}(V_j)\). Let \(\{U^{k,i}\}_{i=1}^{n_k}\) be the open cover
of \(U\) given by the connected components of \(U\cap\phi^{-k}(V_j)\).
Notice that for each \(U^{k,i}\) there is a frame \(A^{k,i}\) of
$\mathcal{A}^1(E^k)$ which is the restriction of  the relevant $C^{k,j}$.

We are going to check that the open cover \(\{U^{k,i}\}_{i=1}^{n_k}\) and the corresponding sections \(\{A^{k,i}\}_{i=1}^{n_k}\) satisfy asymptotic involutivity and exterior regularity.
From the definition of the \(C^{k,j}\)'s we have that
\[
 A^{k,i}_p=C^{\ell_i}_{\phi^k(p)}\circ D\phi^k_p.
\]
To check the compatibility of the cover we first observe that, by standard estimates for dominated splittings, we have that \(\phi^k_p \mathcal{Y}_p\) is converging to \(F\) and so in particular
\(\phi^k_p \mathcal{Y}_p\) is transverse to \(E^0\). Then we can write
\[
 (A^{k,i}_p|_{\mathcal{Y}_p})^{-1}=(D\phi^k_p)^{-1}\circ(C^{\ell_i}_{\phi^k(p)}|_{\phi^k_*\mathcal{Y}_p})^{-1}
\]
which implies
\[
 \|A^{k,i}_p\circ(A^{k,j}_p|_{\mathcal{Y}_p})^{-1}\|=\|C^{\ell_i}_{\phi^k(p)}\circ(C^{\ell_j}_{\phi^k(p)}|_{\phi^k_*\mathcal{Y}_p})^{-1}\|.
\]

The compatibility follows from the fact that \(C_{1}:=C^{\ell_{i}}\) and \(C_{2}:=C^{\ell_{j}}|_{\phi^{k}_{*}\mathcal Y}\) are orthonormal and therefore  \(C_1, C_2:(E^0)^{\perp}\to\mathbb{R}^n\) are isometries.
Let \(v\in\mathbb{R}^n\) and \(u=C_2^{-1}(v)\)
and we write \(u=u_1+u_2\) with \(u_{1}\in E^0\) and \(u_{2}\in(E^0)^{\perp}\). Then  we have
\(C_1(u)=C_1(u_2)\) and \(C_2(u_2)=v\) and  using that \(C_1|_{(E^0)^{\perp}}\) and
\(C_2|_{(E^0)^{\perp}}\) are isometries we have
\[
 \|C_1\circ C_2^{-1}(v)\|=\|C_1(u)\|=\|C_1(u_2)\|=\|u_2\|=\|C_2(u_2)\|=\|v\|
\]
which gives that \(\|C_1\circ C_2^{-1}\|=1\) which then implies the compatibility.

Therefore it remains to prove asymptotic involutivity and exterior regularity. For both cases
we need to estimate
$$
\|(A^{k,i}_p|_{\mathcal{Y}_p})^{-1}\| =\frac{1}{\|A^{k,i}_p|_{\mathcal{Y}_p}\|}.
$$
Since $\mathcal{Y}_p$ is transverse to $E_p$, again by standard estimates for
dominated splittings, there exists a constant $C>0$ such that for all $p \in M$ and for any $v \in \mathcal{Y}_p$ with $|v|=1$
$$
|D\phi^k_pv| \geq Cm(D\phi^k|_{F_p}).
$$
Therefore $||A^{k,i}_p|_{\mathcal{Y}_p}||  \geq Cm(D\phi^k|_{F_p})$ and so
\begin{equation}\label{eq-dyn1}
\|(A^{k,i}_p|_{\mathcal{Y}_p})^{-1}\| \leq \frac{1}{Cm(D\phi^k|_{F_p})}.
\end{equation}
Another common term for both asymptotic involutivity and exterior regularity is \(M^{k,i}\) which we estimate as follows.
For  $X \in \mathbb{R}^n$, $Y \in E^k$ we have
$$
|dA^{k,i}(A^{-k,i}X,Y)| = |dC^{\ell_i}(D\phi^k \circ D\phi^{-k} \circ C^{-1,i}X,D\phi^kY)| = |dC^{\ell_i}(C^{-1,i}X, D\phi^kY)|.
$$
Notice that \(\|dC^{l_{i}}\|\) is uniformly bounded then we can estimate the right hand side by the product of the  norms of the two vectors. Since \(\|C^{-1,i}X\|\) is also uniformly bounded    we have
\[
M^{k,i}=\sup_{X \in \mathbb{R}^n, Y \in E^k}|dA^{k,i}(A^{-k,i}X,Y)|\leq C\|D\phi^{k}|_{E^{k}}\|.
\]
To estimate the last remaining term for the asymptotic involutivity, notice that if $X,Y \in E^k$ then
$$
|dA^{k,i}(X,Y)| = |dC^{\ell_i}(D\phi^kX,D\phi^kY)|\leq C\|D\phi^{k}|_{E^{k}}\|^{2},
$$
which implies that
\[
\|dA^{k,i}|_{E^k}\|\leq C\|D\phi^{k}|_{E}\|^{2}.
\]
So using these last three estimates we have
$$
\|dA^{k,i}|_{E^k}\|_{\infty}\|A^{-k,j}\|_{\infty}e^{\epsilon M^{k,i}} \leq \frac{\|D\phi^{k}|_{E^{k}}\|^{2}}{m(D\phi^{k}|_{F})}e^{\epsilon\|D\phi^{k}|_{E^{k}}\|}.
$$
By the domination, there exists \(r<1\) such that
 for \(k\) large enough we have
  $$\sup_{p\in M}\{\|D\phi^k_p|_{E^{k}_{p}}\|\}  < r^{k}\inf_{p\in M}\{m(D\phi^k|_{F_p})\}.$$
 Moreover by the definition of \(E^{k}\), the linear growth assumption holds also for \(E^{k}\), and
therefore  choosing $\epsilon$ small enough, the right hand side goes to zero as \(k\to\infty\) and so the  asymptotic involutivity is satisfied.

Similarly, to estimate the last remaining term of  the exterior regularity condition, notice that we have
$$
|A^{k,i}(X^k_i -X_i)| = |C^{\ell_i}(D\phi^k(X^k-X)| \leq  \|D\phi^kX^k\| + \|D\phi^kX\|\leq 2\max\{\|D\phi^{k}|_{E^{k}}\|, \|D\phi^{k}|_{E}\|\}
$$
which implies that
$$
|A^{k,i}(X^k_i -X_i)|_{\infty}\|A^{-k,j}\|e^{kM^{k,\ell}}| \leq \frac{2\max\{\|D\phi^{k}|_{E^{k}}\|, \|D\phi^{k}|_{E}\|\}}{m(D\phi^{k}|_{F})}e^{\epsilon\|D\Phi^{k}|_{E^{k}}\|}
$$
which  also goes to $0$ for $\epsilon$ small enough, giving exterior regularity.

\subsection{Proof of Theorem \ref{thm-main}}
The proof consists of checking that
strong asymptotic involutivity and strong exterior regularity imply asymptotic involutivity and exterior regularity respectively, therefore Theorem \ref{thm-main} follows directly from the Main Theorem.

First of all, we replace the covering $\{U^k_i\}_{i=1}^{s_k}$ with the whole neighbourhood $U$ so that $s_k\equiv1$ and the compatibility condition is automatically satisfied. Then $A^k$ simply becomes the matrix formed by the \(1\)-forms $\{\eta^k_1,...,\eta^k_n\}$ given by the strong asymptotic involutivity and $B^k$ becomes the matrix formed by the \(1\)-forms $\{\beta^k_1,...,\beta^k_n\}$ given by the strong exterior regularity assumption. Moreover, by the strong version of asymptotic involutivity and exterior regularity, the sequences of \(1\)-forms  $\eta^k_i$ and $\beta^k_i$ converge to a basis $\eta_i$ and $\beta_i$  of $\mathcal{A}^1(E)$ and therefore  $\|A^k\|_{\infty},\|B^k\|_{\infty}, \|A^{-k}\|_{\infty}, \|B^{-k}\|_{\infty} $ are uniformly bounded in \(  k  \). Then it is easy to check that $||dA^k||_{\infty} \leq C\max_i\{|d\eta^k_i|\}$, $\|dB^k\|_{\infty} \leq C\max_i\{|d\beta^k_i|\}$ for some constant $C>0$ and for all \(  k  \). Therefore from the definition of \(  M^{k}  \)(we omit the
superscript \(  \ell  \) since \(  s_{k}\equiv1  \)) we have (using \(  C  \) as a generic constant)
\[
M^{k}_{A}\leq\|dA^{k}\|\cdot\|A^{-k}\|\leq C\max_i\{|d\eta^k_i|\}
\text{ and }
M^{k}_{B}\leq\|dB^{k}\|\cdot\|B^{-k}\|\leq C\max_i\{|d\beta^k_i|\}.
\]
Moreover, using the fact that \(E^{k}\subset\ker(\eta_{i})\) for all \(i\), we have
$\|d\eta^k_j|_{E^k}\|\leq C\|\eta^k_1\wedge...\wedge d\eta^k_j\|$ which implies that $\|dA^k|_{E^k}\| \leq C\sup_{j}\|\eta^k_1 \wedge...\wedge d\eta^k_j\|$. Combining these observations, we have
\[
\|dA^{k}|_{E^{k}}\|\cdot \|A^{-k}\| e^{\epsilon M^{k}_{A}}\leq C\sup_{j}\|\eta^k_1 \wedge...\wedge d\eta^k_j\| e^{C\epsilon  \max_i\{|d\eta^k_i|\}}
\]
which converges to zero by strong asymptotic involutivity.

For the exterior regularity, we observe that
 for all $k,j$ we have
 \[
 \|B^{k}|_{E}\|=\|(B^{k}-B)|_{E}\|\leq C \max_{j\in \{1,...,n\}}|\beta^k_{j}-\beta_{j}|_{\infty}
 \]
 where the first equality is true because \(E\) annihilates \(  B  \). Combining this with the bound on \(  M^{k}_{B}  \), we get the exterior regularity.

  and
$\sup_{s\in \{1,...,m\}}|B^{k,i}(X^k_{s}-X_{s})|_{\infty} \leq C \sup_{s\in \{1,...,n\}}|\beta^k_{s}-\beta_{s}|_{\infty}$ since the maximal angle between $E^k$ and $E$ is proportional to that between $\mathcal{A}^1(E^k)$ and $\mathcal{A}^1(E)$.
Combining these observations one gets that conditions given in Theorem \ref{thm-main} imply those in the Main Theorem.

\subsection{Proof of Theorem \ref{thm-alt2}}
We will prove Theorem \ref{thm-alt2} assuming Theorem \ref{thm-main}. Recall that $E$ has modulus of continuity $w_1(s)$ and has modulus of continuity $w_{2}(s)$ with respect to the variables $y^{\ell}$, which means that it has some basis of sections $Z_i = \sum_{\ell=1}^mb_{i\ell}(x,y)\frac{\partial}{\partial x^i} + \sum_{j=1}^n c_{ij}(x,y)\frac{\partial}{\partial y^j}$ where $b_{i\ell}$ and $c_{ij}$ have modulus of continuity $w_1(s)$ and have modulus of continuity $w_{2}(s)$ with respect to the variables $y^{\ell}$. We can also find a basis of $E$ of the form
$$
X_i = \frac{\partial}{\partial x^i} + \sum_{j=1}^n a_{ij}(x,y)\frac{\partial}{\partial y^j}.
$$
We claim that $a_{ij}(x,y)$ has modulus of continuity $w_1(s)$ and has modulus of continuity $w_{2}(s)$ with respect to the variables $y^{\ell}$. To prove this claim, we write $X_i$ as linear combinations of $Z_i$ where the coefficients in the combinations are obtained from $b_{i\ell}$ and $c_{ij}$ by summing, dividing (whenever non-zero) and multiplying. These coefficients are $a_{ij}$ and by proposition \ref{prop-moc} they have the same modulus of continuity properties as $b_{i\ell}$ and $c_{ij}$.

Now we are going to prove that each $X_i$ is uniquely integrable. As in section \ref{subsec-uniqueness}, given $X_i$, it suffices to prove uniqueness restricted to the plane $\text{span}\{ \frac{\partial}{\partial x^i}, \frac{\partial}{\
\partial y^1},..., \frac{\partial}{\partial y^n}\}$. Define the 1-forms
$$
\eta_j = dy^j - a_{ij}(x,y)dx^i
$$
so that $X_i \subset \cap_{j=1}^n ker(\eta_j)$.  These 1-forms also have the modulus of continuity properties as above.

To prove that $X_i$ is uniquely integrable, it is sufficient, by Theorem \ref{thm-main}, to prove that $\eta_j$ are exterior regular in this plane.

\begin{lem}\label{prop-extmod}$\eta_j$ are strongly exterior regular.
\end{lem}

\begin{proof} Let $a^{\epsilon}_{ij}(x,y)$ be mollifications of  $a_{ij}$ (see \eqref{eq-mol} in the Appendix). Then define
$$
\eta^{\epsilon}_j =  dy^j - a^{\epsilon}_{ij}(x,y)dx \quad \text{ and } \quad  d\eta^{\epsilon}_j = - \sum_{i=1}^n\frac{\partial a^{\epsilon}_{ij}}{\partial y^i}(x,y) dy^i\wedge dx.
$$
By proposition \ref{prop-moll}
$$
|\frac{\partial a^{\epsilon}_{ij}}{\partial y^i}|_{\infty} \leq K\sup_{|s|\leq \epsilon} \frac{w_2(s)}{s} \quad \text{ and } \quad
| a^{\epsilon}_{ij}-a_{ij}|_{\infty} \leq K\sup_{|s|\leq \epsilon} w_1(s)
$$
for all $i,j$ and $\epsilon$, for some $K>0$. Therefore setting
$$
\eta^{k}_j =  dy^j - a^{\frac{1}{k}}_{ij}(x,y)dx
$$
we have that
$$
|\eta^k_j-\eta_j|_{\infty}e^{|d\eta^k_j|_{\infty}} \leq K\sup_{|s|\leq \epsilon} w_1(s)e^{w_2(s)/s}
$$
which goes to $0$, which gives the strong exterior regularity.
\end{proof}

This completes the proof of Theorem \ref{thm-alt2}.

\subsection{Proof of Theorem \ref{thm-pde}}
Now we will prove that Theorem \ref{thm-alt2} implies Theorem \ref{thm-pde}. To apply Theorem \ref{thm-alt2}, note that the problem of integrating a system of first order PDEs into a problem of integrating a bundle, i.e.
solving the PDE \eqref{eq-pde1} is equivalent to integrating the differential forms
$$
\eta_i = dy^i - \sum_{j=1}^n F^{ij}(x,y)dx^j.
$$
Consider the matrix $\hat{F}^{I}(\xi)$ (see section \ref{sec-intropde} and Theorem \ref{thm-pde} for the definition)
and its inverse $B(\xi)$. Let $E = \cap_{i=1}^m ker(\eta_i)$ so that $\mathcal{A}^1(E) = \text{span}\{\eta_i\}_{i=1}^m$. Let
$$\alpha_{\ell} = \sum_{j=1}^m B_{\ell j}(\xi)\eta_j$$
which forms a basis for $\mathcal{A}(1(E))$. Since $B$ is the inverse of $\hat{F}^{I}(\xi)$ and $\hat{F}^{I}(\xi)$ corresponds to the columns $i_1,...,i_n$ of the matrix $\hat{F}$, the new differential forms are of the form

$$\alpha_{\ell} = d\xi_{i_{\ell}} + \sum_{j\notin {i_1,...,i_n}}^n g_{ij}(\xi)d\xi^j$$
Note that $g_{ij}(\xi)$ are $C^{w_1(s)}$ functions that have modulus of continuity $w_2(s)$ with respect to variables $\{\xi_i\}$ for $i \in \{i_1,...,i_n\}$. Indeed the functions $f_{ij}$ satisfy these properties and so do the entries of the matrix $\hat{F}^{I}(\xi)$. Therefore the matrix
$B(\xi)$ whose entries are formed by taking quotients, products and sums of elements of $\hat{F}^{I}(\xi)$ have the same properties, by Proposition \ref{prop-moc}. Since $g_{ij}$ are obtained by products and sums of entries from $B(\xi)$ and $f_{ij}$ and $1$'s they also have the same property, again by Proposition \ref{prop-moc}. Defining the vector-fields for $\ell \notin {i_1,...,i_n}$

$$X_{\ell} = \frac{\partial}{\partial \xi^{\ell}} + \sum_{j \in {i_1,...,i_n}} g_{\ell j}(\xi)\frac{\partial}{\partial \xi^j}$$
we see that $E$ is spanned by $X_{\ell}$, and so is transverse to $\{\frac{\partial}{\partial \xi^j}\}$ for $j \in {i_1,...,i_n}$.
And in particular with respect to variables $\{\xi^j\}$ for $j \in \{i_1,...,i_n\}$, $E$ has modulus of continuity $w_2(s)$ and in general
it has modulus of continuity $w_1(s)$ where $w_1(s),w_2(s)$ satisfy \eqref{eq-mcont2}. And so by theorem \ref{thm-alt2}, $E$ is uniquely integrable.

This completes the proof of theorem \ref{thm-pde}.

\begin{proof}[Proof of proposition \ref{pro-pde}]
Note that solving the PDE given in the proposition is equivalent to integrating the system of differential forms
$$\eta_i = dy^i - G_i(y^i)\sum_{j=1}^m\frac{\partial H_{i}}{\partial x^j}(x)dx^j.$$
We will prove that these differential forms satisfy strong asymptotic involutivity and strong exterior regularity.

Let $G_i^{\epsilon}$ and $H^{\epsilon}_i$ be mollifications of $H_i$ and $G_i$. Then denote

$$\eta_i^{k} =  dy^i - G^{1/k}_i(y^i)\sum_{j=1}^m\frac{\partial H^{1/k}_{i}}{\partial x^j}(x)dx^j.$$

Note that $|\eta_i^k -\eta_i|_{\infty} \rightarrow 0$. This is because by proposition \ref{prop-moll}, if a function $H$ is $C^1$, and $H^{\epsilon}$ are its mollifications, then derivatives of $H^{\epsilon}$ converge to derivative of $H$. Denoting the differential form $\alpha^k_i = \sum_{j=1}^m\frac{\partial H^{1/k}_{ij}}{\partial x^j}(x)dx^j$, this system can be written as

$$\eta^k_i = dy^i - G_i^{1/k}(y^i)\alpha^k_i.$$
Let $d_x$ be the exterior differentiation with respect to only $x$ coordinates. Then  $\alpha^k_i = d_xH^{1/k}_{i}$, so $d\alpha^k_i = d_x^2(H^{1/k}_i)=0$. Therefore

$$d\eta^k_i = -\frac{\partial G_i^{1/k}}{\partial y^i}(y^i)dy^i \wedge \alpha_i^k.$$
Now lets show strong asymptotic involutivity:

$$\eta^k_1 \wedge ... \wedge \eta^k_n \wedge d\eta^k_{\ell} = -\bigwedge_{i=1}^{n}(dy^i - G_i{1/k}(y^i)\alpha^k_i) \wedge (\frac{\partial G_{\ell}^{1/k}}{\partial y^{\ell}}dy^{\ell} \wedge \alpha_{\ell}^k)$$
Note that $\alpha^k_{\ell}\wedge \alpha^k_{\ell}=0$ therefore the expression above reduces to

$$=-dy^{\ell} \wedge \bigwedge_{i\neq \ell}(dy^i - G_i{1/k}(y^i)\alpha^k_i) \wedge (\frac{\partial G_{\ell}^{1/k}}{\partial y^{\ell}}dy^{\ell} \wedge \alpha_{\ell}^k)$$
However the expression above now contains a term of the form $$dy^{\ell} \wedge \frac{\partial G_{\ell}^{1/k}}{\partial y^{\ell}}dy^{\ell}$$
which is $0$. therefore

$$\eta^k_1 \wedge ... \wedge \eta^k_n \wedge d\eta^k_{\ell} =0.$$
So by Theorem \ref{thm-main} this system is integrable. This proves Proposition \ref{pro-pde}.
\end{proof}

\subsection{Proof of Theorem \ref{thm-ode}}

Now we show how Theorem \ref{thm-pde} implies Theorem \ref{thm-ode}. We want to show uniqueness of solutions for the ODE given in equation \eqref{eq-ode1}. We will prove Theorem \ref{thm-ode} by showing that it is a special case of Theorem \ref{thm-pde} with $m=1$ (so that there is no $j$ index for $\hat{F}$). First note that it can be written as the PDE:
\begin{equation}\label{eq-ode41}
  \frac{\partial y^i}{\partial t} = F^i(t,y)
\end{equation}

Then the matrix $\hat{F}(\xi)$ defined in Theorem \ref{thm-pde}, specialised to the case of Theorem \ref{thm-ode},  is the $n\times (n+1)$ matrix which is obtained by adjoining $n\times n$ identity matrix with the column vector formed from $F(\xi)$. We recall that $\tilde{F}=(F,1)$. The condition $\tilde{F}^{i}(\xi) \neq 0$ with $i=1,2,...,n+1$ in Theorem \ref{thm-ode} is equivalent to the condition $\det(\hat{F}^{I}({\xi})) \neq 0$ with $I =(1,2,...,i-1,i+1,...,n+1)$ in Theorem \ref{thm-pde}. And moreover the condition \ref{eq-mcont4} given in Theorem \ref{thm-pde} is equivalent to the condition \ref{eq-mcont2} given in Theorem \ref{thm-ode}.

This finishes the proof of Theorem \ref{thm-ode}.

\appendix

\section{Appendix}
In the appendix we prove or cite some properties of modulus of continuities and mollifications that we use in section \ref{sec-ex}.
\subsection{Modulus of continuity}

\begin{prop}\label{prop-moc} Let $f,g: U \subset \mathbb{R}^n \rightarrow \mathbb{R}^m$ be function with modulus of continuity $w_f(s)$, $w_g(s)$.
Let $K = \max \{\sup_s{f(s)},\sup_s{g(s)}\}$ and $c = \inf_{s}\{g(s)\}$. Then

\begin{enumerate}
  \item Then $f+g$ has modulus of continuity $w_f(s) + w_g(s)$.
  \item If $K \leq \infty$, $f.g$ has modulus of continuity $w_f(s) + w_g(s)$
  \item If $c>0$ then $\frac{f}{g}$ has modulus of continuity $w_f(s) + w_g(s)$
  \item The function $-xln(x)$ has modulus of continuity  $-xln(x)$ for $0<x< \frac{1}{e}$
  \item The function $x^{\alpha}$ for $0<\alpha<1$ has modulus of continuity $x^{\alpha}$.
  \item Assume $f(x^1,...,x^n)$ has modulus of continuity $w_i(s)$ with respect to each variable.
        Let $w(s)$ be a bounded, convex function such that $w_i(s)\leq w(s)$ for all $s$. Then
        $f(x)$ has modulus of continuity $w(\sqrt 2 s)$.
  \item Assume $f=(f^1(x^1,...,x^n),...,f^m(x^1,...,x^n))$ is such that each $f^i$ has modulus of continuity
  $w_{ij}$ with respect to variable $x^j$.  Let $w_j(s)$ be such that $w_{ij}(s)\leq w_j(s)$ for all $i$.
  Then $f$ has modulus of continuity $w_j(s)$ with respect to variable $x^j$.
\end{enumerate}
\end{prop}

\begin{proof} For the first one
$$|f(x)+g(x) - f(y)-g(y)| \leq |f(x)-f(y)| + |g(x)-g(y)| \leq w_f(|x-y|) + w_g(|x-y|)$$
For the second one
$$|f(x)g(x) - f(y)g(y)| \leq |f(x)g(x)-f(y)g(x)| + |f(y)g(x)-f(y)g(y)|$$
$$\leq K(w_f(|x-y|) + w_g(|x-y|))$$
For the third one note that $\frac{f}{g}= f\frac{1}{g}$ and that
$$|\frac{1}{g}(x)-\frac{1}{g}(y)| = |\frac{g(y)-g(x)}{g(y)g(x)}| \leq \frac{1}{c^2}w(|x-y|)$$
For the fourth and the fifth one see \cite{AgaLak93}. For the sixth one consider the case $f=f(x^1,x^2)$,
$$|f(x^1,x^2) - f(y^1,y^2)| \leq |f(x^1,x^2)-f(y^1,x^2)| + |f(y^1,x^2)-f(y^1,y^2)|$$
$$\leq w_1(|x^1-y^1|) + w_2(|x^2-y^2|) \leq w(|x^1-y^1|+|x^2-y^2|)$$
Since $|x^1-y^1| + |x^2-y^2| \leq \sqrt{2}|(x^1,x^2)-(y^1,y^2)|$ we get
$$|f(x^1,x^2) - f(y^1,y^2)| \leq w(\sqrt(2) |(x^1,x^2)-(y^1,y^2)|)$$
The last one is also almost direct, indeed
$$|f(x^1,...,x^j+t,...,x^n)-f(x^1,...,x^j,...,x^n)|^2 = \sum_{i=1}^m|f^i(..x^j+t...)-f^i(...x^j...)|^2$$
$$\leq \sum_{i=1}^mKw^2_{ij}(|t|) \leq K'w^2_{j}(|t|).$$
\end{proof}

\subsection{Mollifications}

In this section we investigate the modulus of continuity of the standard sequence of mollifications. We recall that,
for a continuous function \(f(x^{1},...,x^{n})\),  the family of mollifiers of \(f\), \(\{f^{\epsilon}\}_{\epsilon>0}\) is defined by
\begin{equation}\label{eq-mol}
f^{\epsilon}(x)=\int_{B(x,\epsilon)}\phi_{\epsilon}(y)f(x-y)dy
\end{equation}
where \(  \phi_{\epsilon}(y)=\epsilon^{-n}e^{\epsilon^{2}/(|y-x|^{2}-\epsilon^{2})}/I_{n}\) for \(|y-x|<\epsilon\) and \(I_{n}=\int_{B(x,\epsilon)}e^{\epsilon^{2}/(|y-x|^{2}-\epsilon^{2})}dy\).
The following properties of mollifications are similar to those used in \cite{S2} but are formulated here in terms of modulus of continuity rather than the H\"older norm.

\begin{prop}\label{prop-moll} Assume $f(x^1,...,x^n)$ is a continuous function with modulus of continuity $w$
and, for \(  j=1,...,n  \), let \(  w_{j}  \) be the modulus of continuity of \(  f  \) with respect to the variable $x^j$.
Then there exists a constants $K>0$ such that for all $j=1,...,n$ and $\epsilon>0$ we have
$$|f^{\epsilon} - f|_{\infty} \leq \frac{K}{\epsilon^n}\int_{|s|\leq \epsilon}s^{n-1}w(s)ds \ \text{ and } \
|\frac{\partial f^{\epsilon}}{\partial x^j} | \leq \frac{K}{\epsilon^{n+1}}\int_{|s|\leq \epsilon}s^{n-1}w_j(s)ds.
$$
\end{prop}

\begin{proof}
For the first inequality we have
\[
|f^{\epsilon}(x)-f(x)| = \int_{B(0,\epsilon)}|\phi_{\epsilon}(x)|f(x-y)-f(x)|dy
\leq |\phi_{\epsilon}|_{\infty}\int_{B(0,\epsilon)}w(|y|)dy.
\]
To get the bound in the statement, we first observe that \(  |\phi_{\epsilon}|\leq\epsilon^{-n}  \) and the bound of
\(  \int_{B(0,\epsilon)}w(|y|)dy  \) follows from a standard change of coordinates by passing to polar  coordinates $(r,\theta_1,...,\theta_{n-1})$ for which the  volume form
is $dV = r^{n-1}f(\theta)drd\theta_1...d\theta_{n-1}$.

For the second inequality, first note that for every \(  j=1,...,n  \) we have
\begin{equation}\label{eq-der0}
\int_{B(0,\epsilon)}\frac{\partial \phi_{\epsilon}}{\partial x^j}=0
\end{equation}
and
\begin{equation}\label{eq-der1}
\frac{\partial \phi_{\epsilon}}{\partial x^j}(x) = \frac{1}{\epsilon^{n+1}}\frac{\partial \phi}{\partial x^j}(\frac{x}{\epsilon})
\end{equation}
where \(\phi(x)=e^{1/(|x|^{2}-1)}\) for \(|x|<1\). Moreover, letting \(  \hat y_{i}=(y_{1},...,y_{i-1},0,y_{i+1},...,y_{n})  \) we have
\begin{equation}\label{eq-der3}
|\int_{B(0,\epsilon)}\frac{\partial \phi_{\epsilon}}{\partial x^j}(y)f(x-\hat y_{i})dy|=0
\end{equation}
as can be seen beywriting the integral as a multiple integral with respect to the coordinates and noticing that with respect to the \(  i  \)'th coordinate, the function \(  f(x-\hat y_{i})  \) is a constant function. Therefore it comes out of the integral and multiplies \(  \int_{B(0,\epsilon)}\frac{\partial \phi_{\epsilon}}{\partial x^j}(y) dy  \) which is equal to zero by \eqref{eq-der0}.
From \eqref{eq-der3}, we can write
\[
|\frac{\partial f^{\epsilon}}{\partial x^i}(x)|=|\int_{B(0,\epsilon)}\frac{\partial \phi_{\epsilon}}{\partial x^j}(y)f(x-y)dy|=|\int_{B(0,\epsilon)}\frac{\partial \phi_{\epsilon}}{\partial x^j}(y)(f(x-y)-f(x-\hat y_{j})dy|
\]
Bounding \(   \frac{\partial \phi_{\epsilon}}{\partial x^j} \) by \(  |D\phi_{\epsilon}|_{\infty}  \) and using the modulus of continuity of \(  f  \) with respect to the \(  j  \)'th coordinate we have
\[|\frac{\partial f^{\epsilon}}{\partial x^i}(x)|\leq |D\phi_{\epsilon}|_{\infty} \int_{B(0,\epsilon)}w_j(|y|)dy\\
\leq  |D\phi|_{\infty}\frac{1}{\epsilon^{n+1}} \int_{|s|\leq \epsilon} s^{n-1}w_j(s)ds
\]
where the last inequality is again achieved by passing to polar coordinates and using equation \eqref{eq-der1}. This finishes the proof of the proposition.
\end{proof}

\end{document}